\documentclass[a4paper,12pt]{amsart}
\usepackage{amsmath,amsfonts,amssymb,latexsym}
\usepackage{hyperref}
\newcommand{\RR}{\mathbb{R}}
\newcommand{\CC}{\mathbb{C}}
\newcommand{\QQ}{\mathbb{Q}}
\newcommand{\NN}{\mathbb{N}}
\newcommand{\ZZ}{\mathbb{Z}}
\newcommand{\EE}{\mathbb{E}}
\newcommand{\Bo}{\mathcal{B}}
\newcommand{\Oo}{\mathcal{O}}

\newcommand{\RE}{ {\rm Re \,} }

\newtheorem{Th}{Theorem}
\newtheorem{Lem}{Lemma}
\newtheorem{Cor}{Corollary}
\newtheorem{Prop}{Proposition}

\theoremstyle{definition}
\newtheorem{Df}{Definition}

\theoremstyle{remark}
\newtheorem{Rem}{Remark}
\newtheorem{Ex}{Example}

\begin{document}


\title[Summability of formal solutions]{Summability of formal solutions of linear partial differential
equations with divergent initial data}

\author{S{\l}awomir Michalik}

\address{Institute of Mathematics
Polish Academy of Sciences\\
P.O. Box 21,
\'Sniadeckich 8,
00-956 Warszawa, Poland}
\address{Faculty of Mathematics and Natural Sciences,
College of Science\\
Cardinal Stefan Wyszy\'nski University\\
W\'oycickiego 1/3,
01-938 Warszawa, Poland}
\email{s.michalik@uksw.edu.pl}
\urladdr{\url{http://www.impan.pl/~slawek}}

\begin{abstract}
We study the Cauchy problem for a general homogeneous linear
partial differential equation in two complex
variables with constant coefficients and with divergent initial data.
We state necessary and sufficient conditions for the summability
of formal power series solutions in terms of properties of divergent Cauchy data.
We consider both the summability in one variable $t$ (with coefficients belonging
to some Banach space of Gevrey series with respect to the second variable $z$)
and the summability in two variables $(t,z)$.
The results are presented in the general framework
of moment-PDEs.
\end{abstract}

\maketitle

\section{Introduction}
The problem of summability of formal solutions of linear PDEs
was mainly studied under the assumption that the Cauchy data
are convergent, see Balser \cite{B5},
Balser and Loday-Richaud \cite{B-L}, Balser and Miyake \cite{B-Mi}, Ichinobe \cite{I},
Lutz, Miyake and Sch\"afke \cite{L-M-S}, Malek \cite{Mal2}, Michalik \cite{Mic,Mic2,Mic5} and Miyake \cite{Miy}.

The case of more general initial data was investigated only for the complex heat equation (see Balser 
\cite{B1,B6}). In \cite{B1} Balser considered the case of entire initial data with an appropriate growth 
condition and he gave some preliminary results for divergent initial data, too.
Next, these results were extended in \cite{B6}, where a characterisation of summable formal power series 
solutions of the complex heat equation in terms of properties of divergent Cauchy data  was given.

The aim of our paper is a generalisation of Balser's results \cite{B1,B6} to homogeneous linear partial 
differential equations with constant coefficients.

Namely, we consider the initial value problem for a general linear partial differential equation with constant coefficients
in two complex variables $(t,z)$
\begin{gather}
\label{eq:general_p}
P(\partial_{t},\partial_{z})\widehat{u}=0,\quad \partial^j_{t} \widehat{u}(0,z)=\widehat{\varphi}_j(z)\quad (j=0,\dots,n-1),
\end{gather}
where $P(\lambda,\zeta)$ is a polynomial in both variables of degree $n$ with respect to $\lambda$ and the
Cauchy data $\widehat{\varphi}_j(z)=\sum_{n=0}^{\infty}\varphi_{jn}z^n\in\CC[[z]]$ are formal power series.
\par
We study the Gevrey asymptotic properties of formal power series solutions $\widehat{u}$ for
a fixed Gevrey order of the initial data.
Moreover, we characterise the multisummable formal solutions $\widehat{u}$ of (\ref{eq:general_p}) in terms of the Cauchy data.
\par
The results are expressed in the general framework of moment differential equations with the differentiations
$\partial_t$ and $\partial_z$ replaced by more general operators of moment differentiations
$\partial_{m_1,t}$ and $\partial_{m_2,z}$ respectively (see Definition \ref{df:moment_diff}).
The general moment differential equations were introduced by Balser and Yoshino \cite{B-Y}, who studied the
Gevrey order of formal solutions of such equations.  A characterisation of
the multisummable formal solutions of moment differential equations in terms of analytic continuation properties 
and growth estimates of the Cauchy data was established in our previous paper \cite{Mic7} under the assumption of 
convergence of the Cauchy data. In the present paper we continue the study without this assumption.
Additionally we consider a wider class of moment functions,
which is a group with respect to multiplication, and so the set of moment differential operators
contains some integro-differential operators (see Example \ref{ex:operator}). 
\par
We give a meaning to summability of formal solutions $\widehat{u}$ in two variables by two methods. In the first 
one we treat $\widehat{u}$ as a formal power series in $t$-variable with the coefficients belonging to some
Banach space of
Gevrey series (in $z$-variable). This situation is carried over by the general theory of summability
developed by Balser \cite{B2}. In the second method we study summability of $\widehat{u}$ in two variables
$(t,z)$ using approaches used by Balser \cite{B6} and by Sanz \cite{S}.
\par
The main idea of the paper is based on the use of appropriate moment Borel transforms $\Bo_{m_1',t}$ and
$\Bo_{m_2',z}$ (see Definition \ref{df:moment_Borel}), which transform the formal solution $\widehat{u}$ of
the equation 
$P(\partial_{m_1,t},\partial_{m_2,z})\widehat{u}=0$
with the divergent Cauchy data $\widehat{\varphi}_j$
into the analytic solution $v=\Bo_{m_1',t}\Bo_{m_2',z}\widehat{u}$ of the equation
$P(\partial_{m_1m_1',t},\partial_{m_2m_2',z})v=0$
with the convergent Cauchy data $\Bo_{m_2',z}\widehat{\varphi}_j$.
On the other hand we are able to define the summability of $\widehat{u}$ (both in $t$ and in $(t,z)$ variables)
in terms of analytic continuation properties of $v$.
In this way, analogously to \cite{Mic7}, we reduce
the problem of summability of $\widehat{u}$ to the problem of analytic continuation of $v$.
\par
In the case of summability of $\widehat{u}$ with respect to $t$-variable, it is sufficient to apply our previous
result \cite[Theorem 3]{Mic7}, which establishes relation between the analytic continuation properties of $v$
(with respect to $t$) and the Cauchy data $\Bo_{m_2',z}\widehat{\varphi}_j$.
In the case of summability of $\widehat{u}$ in two variables $(t,z)$ the situation is more complicated, since we
have to study the analytic continuation properties of $v$ with respect to both variables. To this end we 
characterise the analytic continuation properties
of $v$ in two variables $(t,z)$ in terms of the Cauchy data.
\par
Finally, in both cases we obtain a characterisation of the multisummable formal solution $\widehat{u}$ of moment 
differential equations in the terms of the divergent initial data $\widehat{\varphi}_j$.

\section{Notation}
We use the following notation.
The complex disc in $\CC^n$ with centre at the origin
and radius $r>0$ is denoted by $D^n_r:=\{z\in\CC^n:\ |z|< r\}$.
To simplify notation, we write $D_r$ instead of $D^1_r$. If the radius $r$ is not essential,
then we denote it briefly by $D^n$ (resp. $D$).
\par
\emph{A sector in a direction $d\in\RR$ with an opening $\varepsilon>0$}
in the universal covering space $\widetilde{\CC}$ of $\CC\setminus\{0\}$ is defined by
\[
S_d(\varepsilon):=\{z\in\widetilde{\CC}:\ z=re^{i\theta},\
d-\varepsilon/2<\theta<d+\varepsilon/2,\ r>0\}.
\]
Moreover, if the value of opening angle $\varepsilon$ is not essential, then we denote it briefly by $S_d$.
\par
Analogously, by \emph{a disc-sector in a direction $d\in\RR$ with an opening $\varepsilon>0$ and radius $r>0$}
we mean a domain $\widehat{S}_d(\varepsilon;r):=S_d(\varepsilon)\cup D_r$. If the values of $\varepsilon$ and $r$
are not essential, we write it $\widehat{S}_d$ for brevity (i.e. $\widehat{S}_d=S_d\cup D$).
\par
By $\mathcal{O}(G)$ we understand the space of holomorphic functions on a domain $G\subseteq\CC^n$.
Analogously,
the space of analytic functions of the variables $z_1^{1/\kappa_{1}},\dots,z_n^{1/\kappa_n}$
($(\kappa_1,\dots,\kappa_n)\in\NN^n$) on $G$ is denoted by $\mathcal{O}_{1/\kappa_{1},\dots,1/\kappa_n}(G)$.
More generally, if $\EE$ denotes a Banach space with a norm $\|\cdot\|_{\EE}$, then
by $\Oo(G,\EE)$ (resp. $\Oo_{1/\kappa_{1},\dots,1/\kappa_n}(G,\EE)$) we shall denote the set of all $\EE$-valued  
holomorphic functions (resp. holomorphic functions of the variables $z_1^{1/\kappa_{1}},\dots,z_n^{1/\kappa_n}$) 
on a domain $G\subseteq\CC^n$.
For more information about functions with values in Banach spaces we refer the reader to \cite[Appendix B]{B2}. 
In the paper, as a Banach space $\EE$ we will
take the space of complex numbers $\CC$ (we abbreviate $\Oo(G,\CC)$ to $\Oo(G)$ and
$\Oo_{1/\kappa_{1},\dots,1/\kappa_n}(G,\CC)$ to $\Oo_{1/\kappa_{1},\dots,1/\kappa_n}(G)$)
or the space of Gevrey series $G_{s,1/\kappa}(r)$ (see Definition \ref{df:G_s}).
\par
\begin{Df}
\label{df:growth}
A function $u\in\Oo_{1/\kappa}(\widehat{S}_d(\varepsilon;r),\EE)$
is of \emph{exponential growth of order at most $K\in\RR$ as $t\to\infty$
in $\widehat{S}_d(\varepsilon;r)$} if for any
$\widetilde{\varepsilon}\in(0,\varepsilon)$ and $\widetilde{r}\in(0,r)$ there exist 
$A,B<\infty$ such that
\begin{gather*}
\|u(t)\|_{\EE}<Ae^{B|t|^K} \quad \textrm{for every} \quad t\in \widehat{S}_d(\widetilde{\varepsilon};\widetilde{r}).
\end{gather*}
The space of such functions is denoted by $\Oo_{1/\kappa}^K(\widehat{S}_d(\varepsilon;r),\EE)$.
\par
Analogously, a function $u\in\Oo_{1/\kappa_1,1/\kappa_2}(\widehat{S}_{d_1}(\varepsilon_1;r_1)\times\widehat{S}_{d_2}(\varepsilon_2;r_2))$
is of \emph{exponential growth of order at most $(K_1,K_2)\in\RR^2$ as $(t,z)\to\infty$
in $\widehat{S}_{d_1}(\varepsilon_1;r_1)\times\widehat{S}_{d_2}(\varepsilon_2;r_2)$} if
for any $\widetilde{\varepsilon}_i\in(0,\varepsilon_i)$ and any $\widetilde{r}_i\in(0,r_i)$ ($i=1,2$)
there exist $A,B_1,B_2<\infty$ such that
\begin{gather*}
|u(t,z)|<Ae^{B_1|t|^{K_1}}e^{B_2|z|^{K_2}} \quad \textrm{for every} \quad (t,z)\in 
\widehat{S}_{d_1}(\widetilde{\varepsilon}_1;\widetilde{r}_1)\times
\widehat{S}_{d_2}(\widetilde{\varepsilon}_2;\widetilde{r}_2).
\end{gather*}
The space of such functions is denoted by
$\Oo^{K_1,K_2}_{1/\kappa_1,1/\kappa_2}(\widehat{S}_{d_1}(\varepsilon_1;r_1)\times\widehat{S}_{d_2}(\varepsilon_2;r_2) )$.
\end{Df}
\par
The space of formal power series
$ \widehat{u}(t)=\sum_{j=0}^{\infty}u_j t^{j/\kappa}$ with $u_j\in\mathbb{E}$ is denoted by $\mathbb{E}[[t^{\frac{1}{\kappa}}]]$.
Analogously, the space of formal power series $\widehat{u}(t,z)=\sum_{j,n=0}^{\infty}u_{jn} t^{j/\kappa_1}
z^{n/\kappa_2}$ with
$u_{jn}\in\mathbb{E}$ is denoted by $\mathbb{E}[t^{\frac{1}{\kappa_1}},z^{\frac{1}{\kappa_2}}]]$.
We use the "hat" notation ($\widehat{u}$, $\widehat{v}$, $\widehat{\varphi}$, $\widehat{\psi}$, $\widehat{f}$) to denote the formal power series.
If the formal power series $\widehat{u}$ (resp. $\widehat{v}$, $\widehat{\varphi}$, $\widehat{\psi}$, $\widehat{f}$) is convergent,
we denote its sum by $u$ (resp. $v$, $\varphi$, $\psi$, $f$).

\section{Moment functions}
   In this section we recall the notion of moment methods introduced by Balser \cite{B2}.
   
   \begin{Df}[see {\cite[Section 5.5]{B2}}]
    \label{df:moment}
    A pair of functions $e_m$ and $E_m$ is said to be \emph{kernel functions of order $k$} ($k>1/2$) if
    they have the following properties:
   \begin{enumerate}
    \item[1.] $e_m\in\Oo(S_0(\pi/k))$, $e_m(z)/z$ is integrable at the origin, $e_m(x)\in\RR_+$ for $x\in\RR_+$ and
     $e_m$ is exponentially flat of order $k$ in $S_0(\pi/k)$ (i.e. $\forall_{\varepsilon > 0} \exists_{A,B > 0}$
     such that $|e_m(z)|\leq A e^{-(|z|/B)^k}$ for $z\in S_0(\pi/k-\varepsilon)$).
    \item[2.] $E_m\in\Oo^{k}(\CC)$ and $E_m(1/z)/z$ is integrable at the origin in $S_{\pi}(2\pi-\pi/k)$.
    \item[3.] The connection between $e_m$ and $E_m$ is given by the \emph{corresponding moment function
    $m$ of order $1/k$} as follows.
     The function $m$ is defined in terms of $e_m$ by
     \begin{gather}
      \label{eq:e_m}
      m(u):=\int_0^{\infty}x^{u-1} e_m(x)dx \quad \textrm{for} \quad \RE u \geq 0
     \end{gather}
     and the kernel function $E_m$ has the power series expansion
     \begin{gather}
      \label{eq:E_m}
      E_m(z)=\sum_{n=0}^{\infty}\frac{z^n}{m(n)} \quad  \textrm{for} \quad z\in\CC.
     \end{gather}
   \end{enumerate}
   \end{Df}
   
    Observe that in case $k\leq 1/2$ the set $S_{\pi}(2\pi-\pi/k)$ is not defined,
    so the second property in Definition \ref{df:moment} can not be satisfied. It means that we
    must define the kernel functions of order $k\leq 1/2$ and the corresponding moment functions
    in another way.
    
    \begin{Df}[see {\cite[Section 5.6]{B2}}]
     \label{df:small}
     A function $e_m$ is called \emph{a kernel function of order $k>0$} if we
     can find a pair of kernel functions $e_{\widetilde{m}}$ and $E_{\widetilde{m}}$ of
     order $pk>1/2$ (for some $p\in\NN$) so that
     \begin{gather*}
      e_m(z)=e_{\widetilde{m}}(z^{1/p})/p \quad \textrm{for} \quad z\in S(0,\pi/k).
     \end{gather*}
     For a given kernel function $e_m$ of order $k>0$ we define the
     \emph{corresponding moment function $m$ of order $1/k>0$} by (\ref{eq:e_m}) and
     the \emph{kernel function $E_m$ of order $k>0$} by (\ref{eq:E_m}).
    \end{Df}
    
    \begin{Rem}
     \label{re:m_tilde}
     Observe that by Definitions \ref{df:moment} and \ref{df:small} we have
     \begin{eqnarray*}
      m(u)=\widetilde{m}(pu) & \textrm{and} &
      E_m(z)=\sum_{j=0}^{\infty}\frac{z^j}{m(j)}=\sum_{j=0}^{\infty}\frac{z^j}{\widetilde{m}(jp)}.
     \end{eqnarray*}
    \end{Rem}

We extend the notion of moment functions to real orders as follows
\begin{Df}
 \label{df:moment_general}
     We say that $m$ is a \emph{moment function of order $1/k<0$} if $1/m$ is a moment function of order $-1/k>0$.
     \par
     We say that $m$ is a \emph{moment function of order $0$} if there exist moment functions $m_1$ and $m_2$ of the same order $1/k>0$ such that $m=m_1/m_2$.
\end{Df}

By Definition \ref{df:moment_general} and by \cite[Theorems 31 and 32]{B2} we have
\begin{Prop}
 \label{pr:moments}
 Let $m_1$, $m_2$ be moment functions of orders $s_1,s_2\in\RR$ respectively. Then
 \begin{enumerate}
  \item $m_1m_2$ is a moment function of order $s_1+s_2$,
  \item $m_1/m_2$ is a moment function of order $s_1-s_2$.
 \end{enumerate}
\end{Prop}

\begin{Rem}
 By the above proposition we see that the set $\mathcal{M}$ of all moment functions endowed with the
 multiplication operation has the structure of group $\langle\mathcal{M}, \cdot \rangle$. Moreover,
 the map $\textrm{ord}\,\colon
 \langle\mathcal{M}, \cdot \rangle \longrightarrow \langle\ZZ, +\rangle$ defined by
 $\textrm{ord}\,(m):=s$ for every moment function $m$ of order $s$, is a group homomorphism.
 \end{Rem}

\begin{Ex}
\label{ex:functions}
 For any $a\geq 0$, $b\geq 1$ and $k>0$ we can construct the following examples of kernel functions $e_m$ and 
 $E_m$ of 
orders $k>0$ with the corresponding moment function $m$ of order $1/k$ satisfying Definition \ref{df:moment}
or \ref{df:small}:
 \begin{itemize}
  \item $e_m(z)=akz^{bk}e^{-z^k}$,
  \item $m(u)=a\Gamma(b+u/k)$,
  \item $E_m(z)=\frac{1}{a}\sum_{j=0}^{\infty}\frac{z^j}{\Gamma(b+j/k)}$.
\end{itemize}
In particular for $a=b=1$ we get the kernel functions and the corresponding moment function, which are used  
in the classical theory of $k$-summability.
\begin{itemize}
    \item $e_m(z)=kz^ke^{-z^k}$,
    \item $m(u)=\Gamma(1+u/k)$,
    \item $E_m(z)=\sum_{j=0}^{\infty}z^j/\Gamma(1+j/k)=:\mathbf{E}_{1/k}(z)$, where $\mathbf{E}_{1/k}$ is the
    Mittag-Leffler function of index $1/k$.
\end{itemize}
\end{Ex}

\begin{Ex}
For any $s\in\RR$ we will denote by $\Gamma_s$ the function
\[
  \Gamma_s(u):=\left\{
  \begin{array}{lll}
    \Gamma(1+su) & \textrm{for} & s \geq 0\\
    1/\Gamma(1-su) & \textrm{for} & s < 0.
  \end{array}
  \right.
\]
Observe that by Example \ref{ex:functions} and Definition \ref{df:moment_general}, $\Gamma_s$
is an example of moment function of order $s\in\RR$.
\end{Ex}

The moment functions $\Gamma_s$ will be extensively used in the paper,
since every moment function $m$ of order $s$ has the same growth as $\Gamma_s$. Precisely speaking,
we have 
\begin{Prop}[see {\cite[Section 5.5]{B2}}]
  \label{pr:order}
  If $m$ is a moment function of order $s\in\RR$
  then there exist constants $c,C>0$ such that
    \begin{gather*}
     c^n\Gamma_s(n)\leq m(n) \leq C^n\Gamma_s(n) \quad \textrm{for every} \quad n\in\NN.
    \end{gather*}
  \end{Prop}

\section{Moment Borel transform, Gevrey order and Borel summability}
We use the moment function to define the Gevrey order and the Borel summability. We first introduce
\begin{Df}
 \label{df:moment_Borel}
 Let $\kappa\in\NN$ and $m$ be a moment function. Then the linear operator $\Bo_{m,x^{1/\kappa}}\colon \EE[[x^{\frac{1}{\kappa}}]]\to\EE[[x^{\frac{1}{\kappa}}]]$ defined by
 \[
  \Bo_{m,x^{1/\kappa}}\big(\sum_{j=0}^{\infty}u_jx^{j/\kappa}\big):=
  \sum_{j=0}^{\infty}\frac{u_j}{m(j/\kappa)}x^{j/\kappa}
 \]
 is called an \emph{$m$-moment Borel transform with respect to $x^{1/\kappa}$}.
\end{Df}

We define the Gevrey order of formal power series as follows
   \begin{Df}
    \label{df:summab}
    Let $\kappa\in\NN$ and $s\in\RR$. Then
    $\widehat{u}\in\EE[[x^{\frac{1}{\kappa}}]]$ is called a \emph{formal power series of Gevrey order $s$} if
    there exists a disc $D\subset\CC$ with centre at the origin such that
    $\Bo_{\Gamma_s,x^{1/\kappa}}\widehat{u}\in\Oo_{1/\kappa}(D,\EE)$. The space of formal power series of Gevrey 
    order $s$ is denoted by $\EE[[x^{\frac{1}{\kappa}}]]_s$.
    \par
    Analogously, if $\kappa_1,\kappa_2\in\NN$ and $s_1,s_2\in\RR$ then
    $\widehat{u}\in\EE[[t^{\frac{1}{\kappa_1}},z^{\frac{1}{\kappa_2}}]]$ is called a \emph{formal power series of Gevrey order $(s_1,s_2)$} if there exists a disc $D^2\subset\CC^2$ with centre at the origin such that
    $\Bo_{\Gamma_{s_1},t^{1/\kappa_1}}\Bo_{\Gamma_{s_2},z^{1/\kappa_2}}\widehat{u}\in
    \Oo_{1/\kappa_1,1/\kappa_2}(D^2,\EE)$. The space of formal power series of Gevrey 
    order $(s_1,s_2)$ is denoted by $\EE[[t^{\frac{1}{\kappa_1}},z^{\frac{1}{\kappa_2}}]]_{s_1,s_2}$.
   \end{Df}

\begin{Rem}
 \label{re:moment}
 By Proposition \ref{pr:order}, we may replace $\Gamma_s$ (resp. $\Gamma_{s_1}$ and $\Gamma_{s_2}$) in Definition
 \ref{df:summab} by any moment function $m$ of order $s$ (resp. by any moment functions $m_1$ and $m_2$ of orders
 $s_1$ and $s_2$).
\end{Rem}

\begin{Rem}
 \label{re:convergent}
 If $\widehat{u}\in\EE[[x^{\frac{1}{\kappa}}]]_s$ and $s\leq 0$ then the formal series $\widehat{u}$ is convergent,
 so its sum $u$ is well defined.
 Moreover, $\widehat{u}\in\EE[[x^{\frac{1}{\kappa}}]]_0 \Leftrightarrow u\in\Oo_{1/\kappa}(D,\EE)$ and
 $\widehat{u}\in\EE[[x^{\frac{1}{\kappa}}]]_s \Leftrightarrow u\in\Oo^{-1/s}_{1/\kappa}(\CC,\EE)$ for $s<0$.
\end{Rem}

By Definitions \ref{df:moment_Borel} and \ref{df:summab} we obtain
\begin{Prop}
 \label{pr:properties}
 For every $\widehat{u}\in\EE[[x^{\frac{1}{\kappa}}]]$ the following properties of moment Borel transforms are satisfied:
 \begin{itemize}
  \item $\Bo_{m_1,x^{1/\kappa}}\Bo_{m_2,x^{1/\kappa}}\widehat{u} = \Bo_{m_1m_2,x^{1/\kappa}}\widehat{u}$ for every moment functions $m_1$ and
  $m_2$.
  \item $\Bo_{m,x^{1/\kappa}}\Bo_{1/m,x^{1/\kappa}}\widehat{u}=
  \Bo_{1/m,x^{1/\kappa}}\Bo_{m,x^{1/\kappa}}\widehat{u}=\Bo_{1,x^{1/\kappa}}\widehat{u}=\widehat{u}$ for every moment function $m$.
  \item $\widehat{u}\in\EE[[x^{\frac{1}{\kappa}}]]_{s_1} \Leftrightarrow \Bo_{m,x^{1/\kappa}}\widehat{u}\in\EE[[x^{\frac{1}{\kappa}}]]_{s_1-s}$  for every $s,s_1\in\RR$ and for every moment function $m$ of order $s$.
 \end{itemize}
\end{Prop}

As a Banach space $\EE$ we will take the space of complex numbers $\CC$ or the space of Gevrey series 
$G_{s,1/\kappa}(r)$ defined below.

\begin{Df}
\label{df:G_s}
Fix $\kappa\in\NN$, $r>0$ and $s\in\RR$. By $G_{s,1/\kappa}(r)$ we denote a Banach space of Gevrey series
\[
 G_{s,1/\kappa}(r):=\{\widehat{\varphi}\in\CC[[z^{\frac{1}{\kappa}}]]_s\colon \Bo_{\Gamma_s,z^{1/\kappa}}\widehat{\varphi}\in\Oo_{1/\kappa}(D_r)\cap C(\overline{D_r})\}
\]
equipped with the norm
\[
 \|\widehat{\varphi}\|_{G_{s,1/\kappa}(r)}:=\max_{|z|\leq r}|\Bo_{\Gamma_s,z^{1/\kappa}}\widehat{\varphi}(z)|.
\]
We also set $G_{s,1/\kappa}:=\varinjlim\limits_{r>0}G_{s,1/\kappa}(r)$. Analogously, we define
$\Oo_{1/\widetilde{\kappa}}(G,G_{s,1/\kappa}):=\varinjlim\limits_{r>0}\Oo_{1/\widetilde{\kappa}}(G,G_{s,1/\kappa}(r))$ and $\Oo^K_{1/\widetilde{\kappa}}(G,G_{s,1/\kappa}):=\varinjlim\limits_{r>0}\Oo^K_{1/\widetilde{\kappa}}(G,G_{s,1/\kappa}(r))$.
\par
Moreover, we denote by $G_{s_2,1/\kappa}[[t]]_{s_1}$ the space of
formal power series $\widehat{u}(t,z)=\sum_{j=0}^{\infty}\widehat{u}_j(z)t^j$ of Gevrey order $s_1$ with
coefficients $\widehat{u}_j(z)\in G_{s_2,1/\kappa}$.
\end{Df}

By Definitions \ref{df:summab}, \ref{df:G_s}, Remark \ref{re:moment} and Proposition 
\ref{pr:properties} we conclude
\begin{Prop}
\label{pr:prop2}
 For every $\kappa\in\NN$, $s, \overline{s}\in\RR$ (resp. $s_1,s_2,\overline{s}\in\RR$) and for every moment function $m$ of order $\overline{s}$ the following conditions are equivalent:
 \begin{itemize}
  \item $\widehat{u}\in\CC[[x^{\frac{1}{\kappa}}]]_s$
  (resp. $\widehat{u}\in\CC[[t,z^{\frac{1}{\kappa}}]]_{s_1,s_2}$),
  \item $\Bo_{\Gamma_s,x^{1/\kappa}}\widehat{u}\in\Oo_{1/\kappa}(D)$ (resp. $\Bo_{\Gamma_{s_1},t}
  \Bo_{\Gamma_{s_2},z^{1/\kappa}}\widehat{u}\in\Oo_{1,1/\kappa}(D^2)$),
  \item there exists $r>0$ such that $\widehat{u}\in G_{s,1/\kappa}(r)$ (resp.
  $\widehat{u}\in G_{s_2, 1/\kappa}(r)[[t]]_{s_1}$),
  \item $\widehat{u}\in G_{s,1/\kappa}$ (resp.
  $\widehat{u}\in G_{s_2,1/\kappa}[[t]]_{s_1}$),
  \item $\Bo_{m,x^{1/\kappa}}\widehat{u}\in\CC[[x^{\frac{1}{\kappa}}]]_{s-\overline{s}}$
   (resp. $\Bo_{m,z^{1/\kappa}}\widehat{u}\in G_{s_2-\overline{s},1/\kappa}[[t]]_{s_1}$). 
 \end{itemize}
\end{Prop}

Now we are ready to define the summability of formal power series in one variable (see Balser \cite{B2})
\begin{Df}
\label{df:summable}
Let $\kappa\in\NN$, $K>0$ and $d\in\RR$. Then $\widehat{u}\in\EE[[x^{\frac{1}{\kappa}}]]$ is called
\emph{$K$-summable in a direction $d$} if there exists a disc-sector $\widehat{S}_d$ in a direction $d$ such that
$\Bo_{\Gamma_{1/K},x^{1/\kappa}}\widehat{u}\in\Oo^K_{1/\kappa}(\widehat{S}_d,\EE)$.
\end{Df}

\begin{Rem}
\label{re:summable}
 By Definitions \ref{df:G_s} and \ref{df:summable}, $\widehat{u}\in G_{s,1/\kappa}[[t]]$ is $K$-summable in a direction $d$
 if and only if $\Bo_{\Gamma_{1/K},t}\Bo_{\Gamma_{s},z^{1/\kappa}}\widehat{u}\in\Oo^K_{1,1/\kappa}(\widehat{S}_d\times D)$. Moreover,
 we may replace $\Gamma_s$ in the above characterisation by any moment function $m_2$ of order $s$.
\end{Rem}

We can now define the multisummability in a multidirection.
\begin{Df}
 Let $K_1>\cdots>K_n>0$. We say that a real vector $(d_1,\dots,d_n)\in\RR^n$ is an
 \emph{admissible multidirection} if
 \begin{gather*}
  |d_j-d_{j-1}| \leq \pi(1/K_j - 1/K_{j-1})/2 \quad \textrm{for} \quad j=2,\dots,n.
 \end{gather*}
 \par
 Let $\mathbf{K}=(K_1,\dots,K_n)\in\RR^n_+$ and  let $\mathbf{d}=(d_1,\dots,d_n)\in\RR^n$ be an
 admissible multidirection.
 We say that a formal power series
 $\widehat{u}\in\EE[[x]]$ is {\em $\mathbf{K}$-multisummable in the
 multidirection $\mathbf{d}$}
 if $\widehat{u}=\widehat{u}_1+\cdots+\widehat{u}_n$, where $\widehat{u}_j\in\EE[[x]]$ is
 $K_j$-summable
 in the direction $d_j$ for $j=1,\dots,n$.
\end{Df}

Following Sanz \cite{S} we extend the notion of summability to two variables
\begin{Df}
\label{df:summable2}
For $\kappa_1,\kappa_2\in\NN$, $K_1,K_2>0$ and $d_1,d_2\in\RR$ the formal power series
$\widehat{u}\in\CC[[t^{\frac{1}{\kappa_1}},z^{\frac{1}{\kappa_2}}]]$ is called
\emph{$(K_1,K_2)$-summable in the direction $(d_1,d_2)$} if there exist disc-sectors $\widehat{S}_{d_1}$
and $\widehat{S}_{d_2}$ such that
$\Bo_{\Gamma_{1/K_1},t^{1/\kappa_1}}\Bo_{\Gamma_{1/K_2},z^{1/\kappa_2}}\widehat{u}\in
\Oo_{1/\kappa_1,1/\kappa_2}^{K_1,K_2}(\widehat{S}_{d_1}\times\widehat{S}_{d_2})$.
\end{Df}

\begin{Rem}
 \label{re:sum}
 By the general theory of moment summability (see \cite[Section 6.5 and Theorem 38]{B2}), we may replace
 $\Gamma_{1/K}$  in Definition \ref{df:summable} (resp. $\Gamma_{1/K_1}$ and $\Gamma_{1/K_2}$ 
 in Definition \ref{df:summable2}) by any 
 moment function $m$ of order $1/K$ (resp. by any moment functions $m_1$ of order $1/K_1$ and $m_2$ of order $1/K_2$).
\end{Rem}

More general approach to summability in several variables was given by Balser \cite{B6}. Namely, he introduced
\begin{Df}
 \label{df:Bsum}
 Let $s_1,s_2>0$, 
 $O\subset \{(t_0,z_0)\in(\widetilde{\CC\setminus\{0\}})^2\colon \|(t_0,z_0)\|=1\}$ be bounded, open
 and simply connected and let
 \[
  G=\{(t,z)\in(\widetilde{\CC\setminus\{0\}})^2\colon (t,z)=(x^{s_1}t_0,x^{s_2}z_0),\ 
 (t_0,z_0)\in O,\ x>0\}.
 \]
 Then we say that $G$ is \emph{a $(s_1,s_2)$-region of infinity radius with an opening $O$}.
 
 Moreover, for $\kappa_1,\kappa_2\in\NN$ the formal power series
 \[
 \widehat{u}(t,z)=\sum_{j,n=0}^{\infty}u_{jn}t^{j/\kappa_1}z^{n/\kappa_2}\in
 \CC[[t^{\frac{1}{\kappa_1}},z^{\frac{1}{\kappa_2}}]]
 \]
 is called
 \emph{$(1/s_1,1/s_2)$-summable in the direction $O$} if
 \[
 \Bo_{(s_1,s_2)}\widehat{u}(t,z):=\sum_{j,n=0}^{\infty}\frac{u_{jn}}{\Gamma(1+s_1j/\kappa_1+s_2n/\kappa_2)}
 t^{j/\kappa_1}z^{n/\kappa_2}
 \]
 belongs to the 
 space $\Oo_{1/\kappa_1,1/\kappa_2}(G\cup D^2)$ and for every $O'\Subset O$ there exist $A,B>0$ such that
 \begin{eqnarray*}
  |\Bo_{(s_1,s_2)}\widehat{u}(x^{s_1}t_0,x^{s_2}z_0)| \leq Ae^{Bx} & \textrm{for every} & (t_0,z_0)\in O',\ x>0.
 \end{eqnarray*}
\end{Df}

In the paper we will consider only the situation, when $G$ is a polysector $S_{d_1}\times S_{d_2}$ with
an opening  
\[ 
 O=O_{d_1,d_2}:=\{(t_0,z_0)\in S_{d_1}\times S_{d_2}\colon \|(t_0,z_0)\|=1\}.
\]
In this case, immediately by Definition \ref{df:Bsum}, we get
\begin{Prop}
 \label{pr:Bsum}
 Let $s_1,s_2>0$, $d_1,d_2\in\RR$ and $\kappa_1,\kappa_2\in\NN$. Then the formal power series 
 $\widehat{u}\in\CC[[t^{\frac{1}{\kappa_1}},z^{\frac{1}{\kappa_2}}]]$
 is $(1/s_1,1/s_2)$-summable in the direction $O_{d_1,d_2}$ if and only if
 $\Bo_{(s_1,s_2)}\widehat{u}\in\Oo^{1/s_1,1/s_2}_{1/\kappa_1,1/\kappa_2}(\widehat{S}_{d_1}\times \widehat{S}_{d_2})$.
\end{Prop}

The connection between the Borel type transforms $\Bo_{\Gamma_{s_1,t}}\Bo_{\Gamma_{s_2,z}}$ and $\Bo_{(s_1,s_2)}$
is given in the next lemma.
\begin{Lem}
\label{le:beta}
 Let $s_1,s_2>0$ and $\widehat{u}\in\CC[[t,z]]$. Then the formal power series
 $\widehat{v}(t,z):=\Bo_{\Gamma_{s_1,t}}\Bo_{\Gamma_{s_2,z}}\widehat{u}(t,z)$ and 
 $\widehat{w}(t,z):=\Bo_{(s_1,s_2)}\widehat{u}(t,z)$ are connected by the formula
 \[
  \widehat{w}(t,z)=(1+s_1t\partial_t+s_2z\partial_z)
  \int_0^1\widehat{v}(t\varepsilon^{s_1},z(1-\varepsilon)^{s_2})\,d\varepsilon.
 \]
\end{Lem}
\begin{proof}
 Let $\widehat{u}(t,z)=\sum_{k,n=0}^{\infty}u_{kn}t^kz^n$. 
 Then 
 \begin{gather*}
 \widehat{v}(t,z)=\sum_{k,n=0}^{\infty}\frac{u_{kn}t^kz^n}{\Gamma(1+ks_1)\Gamma(1+ns_2)}\quad \textrm{and}\quad 
 \widehat{w}(t,z)=\sum_{k,n=0}^{\infty}\frac{u_{kn}t^kz^n}{\Gamma(1+ks_1+ns_2)}.
 \end{gather*}
 Using properties of the beta function
   \[
    \int_0^1\varepsilon^{k s_1}(1-\varepsilon)^{ns_2}\,d\varepsilon = B(1+ks_1+ns_2)=
    \frac{\Gamma(1+ks_1)\Gamma(1+ns_2)}{\Gamma(2+ks_1+ns_2)}.
   \]
we conclude that
\begin{multline*}
 \int_0^1 \widehat{v}(t\varepsilon^{s_1},z(1-\varepsilon)^{s_2})\,d\varepsilon\\
 =
 \sum_{k,n=0}^{\infty}\frac{u_{kn}t^kz^n}{\Gamma(1+ks_1)\Gamma(1+ns_2)}\int_0^1\varepsilon^{ks_1}
 (1-\varepsilon)^{ns_2}\,d\varepsilon\\
 =\sum_{k,n=0}^{\infty}\frac{u_{kn}t^kz^n}{\Gamma(2+ks_1+ns_2)}.
\end{multline*}
Hence
\[
   \widehat{w}(t,z)=(1+s_1t\partial_t+s_2z\partial_z)\int_0^1
   \widehat{v}(t\varepsilon^{s_1},z(1-\varepsilon)^{s_2})\,d\varepsilon.
\]
\end{proof}

\begin{Rem}
In Theorem \ref{th:summable} we will show that if $\widehat{u}\in\CC[[t,z^{\frac{1}{\kappa}}]]$ is a formal 
solution of (\ref{eq:th_sum}) then 
\[\Bo_{\Gamma_{s_1},t}\Bo_{\Gamma_{s_2},z^{1/\kappa}}\widehat{u}\in\Oo^{1/s_1,1/s_2}_{1,1/\kappa}
(\widehat{S}_{d_1}\times \widehat{S}_{d_2})\Leftrightarrow
 \Bo_{(s_1,s_2)}\widehat{u}\in\Oo^{1/s_1,1/s_2}_{1,1/\kappa}(\widehat{S}_{d_1}\times \widehat{S}_{d_2}).
\]
In other words, for such $\widehat{u}$ we have the equivalence between $(1/s_1,1/s_2)$-summability in 
the direction $(d_1,d_2)$ (introduced by Sanz) and  $(1/s_1,1/s_2)$-summability in 
the direction $O_{d_1,d_2}$ (introduced by Balser).
In our opinion it should be possible to extend the general theory of moment summability (see Balser \cite[Section 6.5]{B2}) to 
several variables and to show that the above equivalence holds for every formal power series 
$\widehat{u}\in\CC[t,z]]$.  
\end{Rem}

\section{Moment operators}
     In this section we recall the notion of moment differential operators constructed recently by Balser and Yoshino
     \cite{B-Y}. We also extend the concept of moment pseudodifferential operators introduced in our previous
     paper \cite{Mic7}.
     
   \begin{Df}
    \label{df:moment_diff}
    Let $m$ be a moment function. Then the linear operator
    $\partial_{m_,x}\colon\EE[[x]]\to\EE[[x]]$
    defined by
    \[
     \partial_{m,x}\Big(\sum_{j=0}^{\infty}\frac{u_{j}}{m(j)}x^{j}\Big):=
     \sum_{j=0}^{\infty}\frac{u_{j+1}}{m(j)}x^{j}
    \]
    is called the \emph{$m$-moment differential operator $\partial_{m,x}$}.
    \par
    More generally, if $\kappa\in\NN$ then the linear operator 
   $\partial_{m_,x^{1/\kappa}}\colon\EE[[x^{\frac{1}{\kappa}}]]\to\EE[[x^{\frac{1}{\kappa}}]]$
    defined by
    \[
     \partial_{m,x^{1/\kappa}}\Big(\sum_{j=0}^{\infty}\frac{u_{j}}{m(j/\kappa)}x^{j/\kappa}\Big):=
     \sum_{j=0}^{\infty}\frac{u_{j+1}}{m(j/\kappa)}x^{j/\kappa}
    \]
    is called the \emph{$m$-moment $1/\kappa$-fractional differential operator $\partial_{m,x^{1/\kappa}}$}.\end{Df}
  
   \begin{Ex}
    \label{ex:operator}
    Below we present some examples of moment differential operators.
    \begin{itemize}
     \item For $m(u)=\Gamma_1(u)$, the operator $\partial_{m,x}$ coincides with the usual differentiation
     $\partial_x$.
     \item For $m(u)=\Gamma_s(u)$ ($s>0$), the operator $\partial_{m,x}$ satisfies
     \[
      (\partial_{m,x}\widehat{u})(x^s)=\partial^s_x(\widehat{u}(x^s)),
     \]
     where $\partial^s_x$ is the Caputo fractional derivative of order $s$
     defined by
       $$
     \partial^{s}_{x}\Big(\sum_{j=0}^{\infty}\frac{u_{j}}{\Gamma_s(j)}x^{sj}\Big):=
     \sum_{j=0}^{\infty}\frac{u_{j+1}}{\Gamma_s(j)}x^{sj}.$$
     \item For $m(u)\equiv 1$, the corresponding operator $\partial_{m,x}$ satisfies 
     \begin{gather*}
     \partial_{m,x}\widehat{u}(x)=
     \frac{\widehat{u}(x)-u_0}{x} \quad \textrm{for every} \quad \widehat{u}(x)=\sum_{j=0}^{\infty}u_jx^j\in\EE[[x]].
     \end{gather*}
     \item For $m(u)=\Gamma_{-1}(u)$, the operator $\partial_{m,x}$ satisfies
     \begin{gather*}
     \partial_{m,x}\widehat{u}(x)=\frac{1}{x}\int_0^x\frac{\widehat{u}(y)-u_0}{y}\,dy \quad \textrm{for every} \quad
     \widehat{u}(x)=\sum_{j=0}^{\infty}u_jx^j\in\EE[[x]].
     \end{gather*}
     \item For $m(u)=\Gamma_{-s}(u)$ ($s>0$), the operator $\partial_{m,x}$ satisfies
     \begin{gather*}
     (\partial_{m,x}\widehat{u})(x^s)=\frac{1}{x^s}\partial^{-s}_x\frac{\widehat{u}(x^s)-u_0}{x^s}
     \quad \textrm{for every} \quad
     \widehat{u}(x)=\sum_{j=0}^{\infty}{u_j}x^j\in\EE[[x]],
     \end{gather*}
     where
     $\partial^{-s}_x$ is the right-inversion operator to $\partial^s_x$ and is defined by
     \[
     \partial^{-s}_{x}\Big(\sum_{j=0}^{\infty}\frac{u_{j}}{\Gamma_s(j)}x^{sj}\Big):=
     \sum_{j=1}^{\infty}\frac{u_{j-1}}{\Gamma_s(j)}x^{sj}.
     \]
\end{itemize}
   \end{Ex}
 \par
 The moment differential operator $\partial_{m,z}$ is well-defined for every $\varphi\in\Oo(D)$. In addition,
 we have the following integral representation of $\partial^n_{m,z}\varphi$.
\begin{Prop}[see {\cite[Proposition 3]{Mic7}}]
 Let $\varphi\in \Oo(D_r)$ and $m$ be a moment function of order $1/k>0$. Then for every $|z|<\varepsilon<r$ and 
 $n\in\NN$ we have
 \[
  \partial_{m,z}^n\varphi(z) = \frac{1}{2\pi i} \oint_{|w|=\varepsilon} \varphi(w)
  \int_0^{\infty(\theta)}\zeta^{n}E_m(z\zeta)\frac{e_m(w\zeta)}{w\zeta}\,d\zeta\,dw,
 \]
 where $\theta\in (-\arg w-\frac{\pi}{2k}, -\arg w + \frac{\pi}{2k})$.
\end{Prop}

Using the above formula, we have defined in \cite[Definition 8]{Mic7} a moment
pseudodifferential operator $\lambda(\partial_{m,z})\colon \Oo(D)\to\Oo(D)$ as an operator satisfying
\begin{eqnarray*}
  \lambda(\partial_{m,z}) E_m(\zeta z):=\lambda(\zeta)E_m(\zeta z)&\textrm{for}& |\zeta|\geq r_0.
\end{eqnarray*}
Namely, if $\lambda(\zeta)$ is 
an analytic function for $|\zeta|\geq r_0$ then $\lambda(\partial_{m,z})$ is defined by
\[
\lambda(\partial_{m,z})\varphi(z):=\frac{1}{2\pi i} \oint_{|w|=\varepsilon} \varphi(w)
  \int_{r_0e^{i\theta}}^{\infty(\theta)}\lambda(\zeta)E_m(\zeta z)\frac{e_m(\zeta w)}{\zeta w}\,d\zeta\,dw
\]
for every $\varphi\in\Oo(D_r)$ and $|z|<\varepsilon < r$, where
$\theta\in (-\arg w-\frac{\pi}{2k}, -\arg w + \frac{\pi}{2k})$.
\par
We extend this definition to the case where $\lambda(\zeta)$ is an analytic function of the variable
$\xi=\zeta^{1/\kappa}$ for $|\zeta|\geq r_0$ (for some $\kappa\in\NN$ and $r_0>0$.
Since 
$(\partial_{m,z}\varphi)(z^{\kappa})=\partial_{\widetilde{m},z}^{\kappa}(\varphi(z^{\kappa}))$ for every $\varphi\in\Oo(D)$, where $\widetilde{m}(u):=m(u/\kappa)$ (see \cite[Lemma 3]{Mic7}),
the operator $\lambda(\partial_{m,z})$ should satisfy the formula
\begin{eqnarray}
 \label{eq:kappa}
 (\lambda(\partial_{m,z})\varphi)(z^{\kappa})=\lambda(\partial^{\kappa}_{\widetilde{m},z})(\varphi(z^{\kappa})) &
 \textrm{for every} & \varphi\in\Oo_{1/\kappa}(D).
\end{eqnarray}
For this reason we have
\begin{Df}
\label{df:pseudo_1}
Let $m$ be a moment function of order $1/k>0$ and
 $\lambda(\zeta)$ be an analytic function of the variable
$\xi=\zeta^{1/\kappa}$ for $|\zeta|\geq r_0$
(for some $\kappa\in\NN$ and $r_0>0$) of polynomial growth at infinity.
A \emph{moment pseudodifferential operator}
 $\lambda(\partial_{m,z})\colon\Oo_{1/\kappa}(D)\to\Oo_{1/\kappa}(D)$ (or, more generally, $\lambda(\partial_{m,z})\colon\EE[[z^{\frac{1}{\kappa}}]]_0\to\EE[[z^{\frac{1}{\kappa}}]]_0$) is defined by
 \begin{equation}
  \label{eq:lambda}
  \lambda(\partial_{m,z})\varphi(z):=\frac{1}{2\kappa\pi i} \oint^{\kappa}_{|w|=\varepsilon}\varphi(w)
  \int_{r_0e^{i\theta}}^{\infty(\theta)}\lambda(\zeta)
  E_{\widetilde{m}}(\zeta^{1/\kappa} z^{1/\kappa})\frac{e_m(\zeta w)}{\zeta w}\,d\zeta\,dw
 \end{equation}
for every $\varphi\in\Oo_{1/\kappa}(D_r)$ and $|z|<\varepsilon < r$, where $\widetilde{m}(u):=m(u/\kappa)$, $E_{\widetilde{m}}(\zeta^{1/\kappa} z^{1/\kappa})=\sum_{n=0}^{\infty}\frac{\zeta^{n/\kappa}z^{n/\kappa}}{\widetilde{m}(n)}$, $\theta\in (-\arg w-\frac{\pi}{2k}, -\arg w + \frac{\pi}{2k})$ and $\oint_{|w|=\varepsilon}^{\kappa}$ means that we integrate $\kappa$ times along the positively oriented circle of radius $\varepsilon$. Here the integration in the inner integral is taken over a ray $\{re^{i\theta}\colon r\geq r_0\}$.
\end{Df}

Observe that
\begin{multline*}
(\lambda(\partial_{m,z})\varphi)(z^{\kappa})=\frac{1}{2\kappa\pi i}\oint_{|w|=\varepsilon}^{\kappa} \varphi(w)
  \int\limits_{r_0e^{i\theta}}^{\infty(\theta)}\lambda(\zeta)E_{\widetilde{m}}(\zeta^{1/\kappa}z)
  \frac{e_m(\zeta w)}{\zeta w}\,d\zeta\,dw\\
  =\frac{1}{2\pi i}\oint_{|w^{\kappa}|=\varepsilon}\varphi(w^{\kappa})
  \int\limits_{r_0^{1/\kappa}e^{i\theta/\kappa}}^{\infty(\theta/\kappa)}\lambda(\zeta^k)
  E_{\widetilde{m}}(\zeta z)\frac{e_{\widetilde{m}}(\zeta w)}{\zeta w}\,d\zeta\,dw=
 \lambda(\partial^{\kappa}_{\widetilde{m},z})(\varphi(z^{\kappa})),
\end{multline*}
so (\ref{eq:kappa}) holds for the operators $\lambda(\partial_{m,z})$ defined by (\ref{eq:lambda}).
\par
Immediately by the definition, we obtain the following connection between the moment Borel transform and the moment
differentiation.
\begin{Prop}
\label{pr:commutation}
Let $m$ and $m'$ be moment functions. Then
the operators $\Bo_{m',x}, \partial_{m,x}\colon\EE[[x]]\to\EE[[x]]$ satisfy the following commutation formulas for every $\widehat{u}\in\EE[[x]]$ and for $\overline{m}=mm'$:
\begin{enumerate}
\item[i)] $\Bo_{m',x}\partial_{m,x}\widehat{u}=\partial_{\overline{m},x}\Bo_{m',x}\widehat{u}$,
\item[ii)] $\Bo_{m',x}P(\partial_{m,x})\widehat{u}=P(\partial_{\overline{m},x})\Bo_{m',x}\widehat{u}$
for any polynomial $P$ with constant coefficients.
\end{enumerate}
\end{Prop}

The same commutation formula holds if we replace $P(\partial_{m,x})$ by $\lambda(\partial_{m,x})$. Namely, we have
\begin{Prop}
\label{pr:commutation2}
Let $m$ and $m'$ be moment functions and $\lambda(\zeta)$ be an analytic function of the variable
$\xi=\zeta^{1/\kappa}$ for $|\zeta|\geq r_0$
(for some $\kappa\in\NN$ and $r_0>0$) of polynomial growth at infinity. Then
the operators $\Bo_{m',x^{1/\kappa}}, \lambda(\partial_{m,x})\colon\EE[[x^{1/\kappa}]]_0\to\EE[[x^{1/\kappa}]]_0$ 
satisfy the commutation formula
$$\Bo_{m',x^{1/\kappa}}\lambda(\partial_{m,x})\widehat{u}=\lambda(\partial_{\overline{m},x})
\Bo_{m',x^{1/\kappa}}\widehat{u}$$
for every $\widehat{u}\in\EE[[x^{1/\kappa}]]_0$ and for $\overline{m}=mm'$.
\end{Prop}
\begin{proof}
 Note that, by Proposition \ref{pr:moments}, $\overline{m}=mm'$ is a moment function.
 Observe that by Definition \ref{df:pseudo} we have
 \begin{align*}
  &\Bo_{m',x^{1/\kappa}}\lambda(\partial_{m,x})\widehat{u}(x)\\
  =&\frac{1}{2\kappa\pi i} \oint^{\kappa}_{|w|=\varepsilon} u(w)
  \int_{r_0e^{i\theta}}^{\infty(\theta)}\lambda(\zeta)\Bo_{m',x^{1/\kappa}}
  E_{\widetilde{m}}(\zeta^{1/\kappa} x^{1/\kappa})
  \frac{e_{m}(\zeta w)}{\zeta w}\,d\zeta\,dw\\
  =&\frac{1}{2\kappa\pi i} \oint^{\kappa}_{|w|=\varepsilon} u(w)
  \int_{r_0e^{i\theta}}^{\infty(\theta)}\lambda(\zeta)E_{\widetilde{\overline{m}}}(\zeta^{1/\kappa} x^{1/\kappa})
  \frac{e_{m}(\zeta w)}{\zeta w}\,d\zeta\,dw\\
  =&\lambda(\partial_{\overline{m},x})\frac{1}{2\kappa\pi i} \oint^{\kappa}_{|w|=\varepsilon} u(w)
  \int_{r_0e^{i\theta}}^{\infty(\theta)}\Bo_{m',x^{1/\kappa}}E_{\widetilde{m}}(\zeta^{1/\kappa} x^{1/\kappa})
  \frac{e_{m}(\zeta w)}{\zeta w}\,d\zeta\,dw\\
  =&\lambda(\partial_{\overline{m},x})\Bo_{m',x^{1/\kappa}}\widehat{u}(x),
 \end{align*}
 where $\widetilde{m}(u):=m(u/\kappa)$ and
 $\widetilde{\overline{m}}(u):=\overline{m}(u/\kappa)=m(u/\kappa)m'(u/\kappa)$.
\end{proof}

Using Proposition \ref{pr:commutation2} we are able to extend Definition \ref{df:pseudo_1} to the formal power series and to the moment functions of real orders. 
\begin{Df}
\label{df:pseudo}
 Let $s\in\RR$, $m$ be a moment function of order $\widetilde{s}\in\RR$ and 
 $\lambda(\zeta)$ be an analytic function of the variable
$\xi=\zeta^{1/\kappa}$ for $|\zeta|\geq r_0$
of polynomial growth at infinity.
A \emph{moment pseudodifferential operator
 $\lambda(\partial_{m,z})$} for the formal power series $\widehat{\varphi}\in\EE[[z^{\frac{1}{\kappa}}]]_s$
 is defined by
 \[
  \lambda(\partial_{m,z})\widehat{\varphi}(z):=\Bo_{\Gamma_{-\overline{s}},z^{1/\kappa}}\lambda(\partial_{\overline{m},z})
  \Bo_{\Gamma_{\overline{s}},z^{1/\kappa}}\widehat{\varphi}(z),
 \]
 where $\overline{m}=m\Gamma_{\overline{s}}$, $\overline{s}=\max\{s,\widetilde{s}+1\}$
 and the operator
 $\lambda(\partial_{\overline{m},z})$ is constructed in Definition \ref{df:pseudo_1}.
\end{Df}
  
\begin{Df}[{\cite[Definition 9]{Mic7}}]
   We define a \emph{pole order $q\in\QQ$} and a \emph{leading term $\lambda\in\CC\setminus\{0\}$}
   of $\lambda(\zeta)$ as the numbers satisfying the formula
   $\lim_{\zeta\to\infty}\lambda(\zeta)/\zeta^q=\lambda$.
   We write it also $\lambda(\zeta)\sim\lambda\zeta^q$.
\end{Df}

At the end of the section we improve the estimate given in \cite[Lemma 1]{Mic7} as follows
\begin{Lem}
\label{le:estimation}
Let $\widehat{\varphi}\in\CC[[z^{\frac{1}{\kappa}}]]_s$, $s\leq 0$, $m$ be a moment function of order $1/k>0$ and $\lambda(\partial_{m,z})$ be a moment pseudodifferential
operator with $\lambda(\zeta)\sim\lambda\zeta^q$ and $q\in\QQ$.
Then there exist $r>0$ and $A,B<\infty$ such that
\begin{gather*}
 \sup_{|z|<r}|\lambda^j(\partial_{m,z})\varphi(z)|\leq |\lambda|^jAB^{j}\Gamma_{\overline{q}(s+1/k)}(j)\quad \textrm{for} 
 \quad
 j=0,1,\dots,
\end{gather*}
where $\overline{q}:=\max\{0,q\}$.
\end{Lem}
\begin{proof}
 Repeating the proof of \cite[Lemma 1]{Mic7}, we may take $r>0$ and $\varepsilon_r>0$ such that
 \[
 \sup_{|z|<r}|\lambda^j(\partial_{m,z})\varphi(z)|\leq |\lambda|^j A_1 B_1^j
 \frac{\Gamma_{\overline{q}/k}(j)}{\varepsilon^{j\overline{q}}}
 \frac{1}{2\kappa\pi\varepsilon}\oint_{|w|=\varepsilon}^{\kappa}|\varphi(w)|\,d|w|
 \]
 for some $A_1,B_1 <\infty$ and for every $\varepsilon > \varepsilon_r$ such that 
 $D_{\varepsilon}\Subset D$ and $\varphi\in\Oo_{1/\kappa}(D)$.
 \par
 If $s=0$ then the assertion is given by the estimation
 \[
  \frac{1}{2\kappa\pi\varepsilon}\oint_{|w|=\varepsilon}^{\kappa}|\varphi(w)|\,d|w|
  \leq A_2.
 \]
 \par
 If $s<0$ then $\varphi\in\Oo^{-1/s}_{1/\kappa}(\CC)$. So we estimate
 \begin{gather*}
  \frac{1}{2\kappa\pi\varepsilon}\oint_{|w|=\varepsilon}^{\kappa}|\varphi(w)|\,d|w|
  \leq A_2e^{B_2\varepsilon^{-1/s}} \quad \textrm{for every} \quad \varepsilon > \varepsilon_r.
 \end{gather*}
 Hence, putting $\varepsilon=(\frac{-sj\overline{q}}{B_2})^{-s}$  and applying the Stirling formula
 (see \cite[Theorem 68]{B2}) we conclude that
 \[
   \sup_{|z|<r}|\lambda^j(\partial_{m,z})\varphi(z)|\leq 
   \frac{|\lambda|^j\widetilde{A}\widetilde{B}^j\Gamma_{\overline{q}/k}(j)e^{-sj\overline{q}}}{(-sj\overline{q})^{-sj\overline{q}}}\leq |\lambda|^j A B^j \Gamma_{\overline{q}(s+1/k)}(j).
 \]
\end{proof}

\section{Formal solutions and Gevrey estimates}
   In this section we study the formal solutions of the initial value problem for a general linear moment partial
   differential equation with constant coefficients
   \begin{equation}
    \label{eq:general}
    \left\{
    \begin{array}{ll}
     P(\partial_{m_1,t},\partial_{m_2,z})\widehat{u}=0&\\
     \partial^j_{m_1,t} \widehat{u}(0,z)=\widehat{\varphi}_j(z)\in\CC[[z]]
     & (j=0,\dots,n-1),
    \end{array}
    \right.
   \end{equation}
   where $m_1$, $m_2$ are moment functions of orders $s_1,s_2\in\RR$ respectively, and
    \begin{equation}
     \label{eq:polynomial}
     P(\lambda,\zeta)=P_0(\zeta)\lambda^n-\sum_{j=1}^nP_j(\zeta)\lambda^{n-j}
    \end{equation}
    is a general polynomial of two variables, which is of order $n$ with respect to $\lambda$.
    
    First, we will show the following
    \begin{Prop}
     \label{pr:formal}
     Let $m'_1$ and $m'_2$ be moment functions, $\widehat{u}\in\CC[[t,z]]$ and
     $\widehat{v}=\Bo_{m'_1,t}\Bo_{m'_2,z}\widehat{u}$. Then
     $\widehat{u}$ is a formal solution of (\ref{eq:general})
     if and only if  $\widehat{v}$ is a formal solution of
    \begin{equation}
    \label{eq:general_2}
    \left\{
    \begin{array}{lll}
     P(\partial_{\overline{m}_1,t},\partial_{\overline{m}_2,z})\widehat{v}=0&&\\
     \partial^j_{\overline{m}_1,t} \widehat{v}(0,z)=\widehat{\psi}_j(z):=\Bo_{m'_2,z}\widehat{\varphi}_j(z)\in\CC[[z]]
     & \textrm{for} & j=0,\dots,n-1,
    \end{array}
    \right.
   \end{equation}
    where $\overline{m}_1:=m_1m_1'$ and $\overline{m}_2:=m_2m_2'$.
  \end{Prop}
  \begin{proof}
   ($\Longrightarrow$)
   We assume that $\widehat{u}$ is a formal solution of (\ref{eq:general}). By Proposition \ref{pr:commutation} we have
   \begin{align*}
   P(\partial_{\overline{m}_1,t},\partial_{\overline{m}_2,z})\widehat{v}&=P(\partial_{\overline{m}_1,t},\partial_{\overline{m}_2,z})\Bo_{m'_1,t}\Bo_{m'_2,z}\widehat{u}\\
    &=\Bo_{m'_1,t}\Bo_{m'_2,z}P(\partial_{m_1,t},\partial_{m_2,z})\widehat{u}=0
   \end{align*}
   and
   \begin{align*}
    \partial^j_{\overline{m}_1,t} \widehat{v}(0,z)
    &= \partial^j_{\overline{m}_1,t}\Bo_{m'_1,t}\Bo_{m'_2,z}\widehat{u}(0,z)
    =\Bo_{m'_1,t}\Bo_{m'_2,z}\partial^j_{m_1,t}\widehat{u}(0,z)\\
    &=\Bo_{m'_2,z}\widehat{\varphi}_j(z)
   \end{align*}
   for $j=0,\dots,n-1$.
   So $\widehat{v}$ is a formal solution of (\ref{eq:general_2}).
   \par
   ($\Longleftarrow$) Observe that $\widehat{u}=\Bo_{1/m_1',t}\Bo_{1/m_2',z}\widehat{v}$ and $\widehat{\varphi}_j=\Bo_{1/m_2',z}\widehat{\psi}_j$
   for $j=0,\dots,n-1$. Repeating the first part of the proof with $\widehat{u}$ replaced by $\widehat{v}$ and
   $\widehat{\varphi}_j$ replaced by $\widehat{\psi}_j$, we obtain the assertion.
  \end{proof}
    If $P_0(\zeta)$ defined by (\ref{eq:polynomial}) is not a constant, then a formal solution of (\ref{eq:general})
    is not uniquely determined. To avoid this inconvenience we choose some special solution which is already 
    uniquely determined. To this end we factorise the moment differential operator
   $P(\partial_{m_1,t},\partial_{m_2,z})$ as follows
   \begin{align*}
    P(\partial_{m_1,t},\partial_{m_2,z})
    &=P_0(\partial_{m_2,z})(\partial_{m_1,t}-\lambda_1(\partial_{m_2,z}))^{n_1}\cdots
    (\partial_{m_1,t}-\lambda_l(\partial_{m_2,z}))^{n_l}\\
    &=:P_0(\partial_{m_2,z})\widetilde{P}(\partial_{m_1,t},\partial_{m_2,z}),
    \end{align*}
   where $\lambda_1(\zeta),\dots,\lambda_l(\zeta)$ are the roots of the characteristic equation
   \linebreak
   $P(\lambda,\zeta)=0$ with multiplicity
   $n_1,\dots,n_l$ ($n_1+\cdots+n_l=n$) respectively.
   \par
     Since $\lambda_{\alpha}(\zeta)$ are algebraic functions,
     we may assume that there exist $\kappa\in\NN$ and $r_0<\infty$ such that
   $\lambda_{\alpha}(\zeta)$ are holomorphic functions of the variable $\xi=\zeta^{1/\kappa}$
   (for $|\zeta|\geq r_0$) and, moreover, there exist $\lambda_{\alpha}\in\CC\setminus\{0\}$ and 
   $q_{\alpha}=\mu_{\alpha}/\nu_{\alpha}$
   (for some relatively prime numbers $\mu_{\alpha}\in\ZZ$ and $\nu_{\alpha}\in\NN$) such that
   $\lambda_{\alpha}(\zeta)\sim\lambda_{\alpha}\zeta^{q_{\alpha}}$ for $\alpha=1,\dots,l$.
   Hence the moment pseudodifferential operators $\lambda_{\alpha}(\partial_{m_2,z})$ are well-defined.
   \par
   Under the above assumption, by a \emph{normalised formal solution} $\widehat{u}$ of (\ref{eq:general}) we mean such solution
   of (\ref{eq:general}), which is also a solution of the pseudodifferential equation
   $\widetilde{P}(\partial_{m_1,t},\partial_{m_2,z})\widehat{u}=0$ (see \cite[Definition 10]{Mic7}).

  Now we are ready to study the Gevrey order of formal solution $\widehat{u}$ of (\ref{eq:general}),
  which depends on the orders $s_1,s_2\in\RR$ of the moment functions $m_1$, $m_2$ respectively, on the Gevrey
  order $s\in\RR$ of the initial data $\widehat{\varphi}$ and depends on the pole orders $q_{\alpha}\in\QQ$ of
  the roots $\lambda_{\alpha}(\zeta)$ ($\alpha=1,\dots,l$).
  We generalise the results for the analytic Cauchy data given in \cite[Theorems 1 and 2]{Mic7}
  as follows
\begin{Th}
  \label{th:gevrey}
  Let $s\in\RR$ and let $\widehat{u}$ be a normalised formal solution of
   (\ref{eq:general}) with $\widehat{\varphi}_j\in\CC[[z]]_s$
   ($j=0,...,n-1$) then $\widehat{u}=\sum_{\alpha=1}^l\sum_{\beta=1}^{n_{\alpha}}\widehat{u}_{\alpha\beta}$ with
   $\widehat{u}_{\alpha\beta}$ being a formal solution of simple pseudodifferential equation
   \begin{equation}
   \label{eq:gevrey}
    \left\{
    \begin{array}{l}
     (\partial_{m_1,t}-\lambda_{\alpha}(\partial_{m_2,z}))^{\beta} \widehat{u}_{\alpha\beta}=0\\
     \partial_{m_1,t}^j \widehat{u}_{\alpha\beta}(0,z)=0\ \ (j=0,\dots,\beta-2)\\
     \partial_{m_1,t}^{\beta-1} \widehat{u}_{\alpha\beta}(0,z)=\lambda_{\alpha}^{\beta-1}(\partial_{m_2,z})
     \widehat{\varphi}_{\alpha\beta}(z),
    \end{array}
    \right.
   \end{equation}
   where $\widehat{\varphi}_{\alpha\beta}(z):=\sum_{j=0}^{n-1}d_{\alpha\beta j}(\partial_{m_2,z})
   \widehat{\varphi}_j(z)\in\CC[[z^{\frac{1}{\kappa}}]]_s$ and $d_{\alpha\beta j}(\zeta)$ are some holomorphic
   functions of the variable $\xi=\zeta^{1/\kappa}$ and of polynomial growth.
   \par
   Moreover, if $q_{\alpha}$ is a pole order of $\lambda_{\alpha}(\zeta)$
   and $\overline{q}_{\alpha}=\max\{0,q_{\alpha}\}$,
   then a formal solution $\widehat{u}_{\alpha\beta}$ is a Gevrey series of order $\overline{q}_{\alpha}(s_2+s) - s_1$
   with respect to $t$. More precisely,
   $\widehat{u}_{\alpha\beta}\in\CC[[t,z^{\frac{1}{\kappa}}]]_{\overline{q}_{\alpha}(s_2+s) - s_1,s}$
   or, equivalently, $\widehat{u}_{\alpha\beta}\in G_{s,1/\kappa}[[t]]_{\overline{q}_{\alpha}(s_2+s) - s_1}$.
\end{Th}
\begin{proof}
 For fixed $\overline{s}>\max\{s,-s_2\}$ we define $\widehat{v}:=\Bo_{\Gamma_{\overline{s}},z}\widehat{u}$.
 By Proposition \ref{pr:formal}, $\widehat{v}$ is a formal solution of
\[
    \left\{
    \begin{array}{lll}
     P(\partial_{m_1,t},\partial_{\overline{m}_2,z})\widehat{v}=0&&\\
     \partial^j_{m_1,t} \widehat{v}(0,z)=\widehat{\psi}_j(z)=\Bo_{\Gamma_{\overline{s}},z}\widehat{\varphi}_j(z)
     \in\CC[[z]]_{s-\overline{s}}
     & \textrm{for} & j=0,\dots,n-1,
    \end{array}
    \right.
\]
where $\overline{m}_2:=m_2\Gamma_{\overline{s}}$.
Since $\overline{m}_2$ is a moment function of order $\overline{s}+s_2>0$ and $\widehat{\psi}_j$
are the Gevrey series of order $s-\overline{s}<0$ for $j=0,\dots,n-1$, repeating the proof of
\cite[Theorem 1]{Mic7} we conclude that
 $\widehat{v}=\sum_{\alpha=1}^l\sum_{\beta=1}^{n_{\alpha}}\widehat{v}_{\alpha\beta}$ with $\widehat{v}_{\alpha\beta}$ being
 a formal solution of
\[
    \left\{
    \begin{array}{l}
     (\partial_{m_1,t}-\lambda_{\alpha}(\partial_{\overline{m}_2,z}))^{\beta} \widehat{v}_{\alpha\beta}=0\\
     \partial_{m_1,t}^j \widehat{v}_{\alpha\beta}(0,z)=0\ \ (j=0,\dots,\beta-2)\\
     \partial_{m_1,t}^{\beta-1} \widehat{v}_{\alpha\beta}(0,z)=\lambda_{\alpha}^{\beta-1}(\partial_{\overline{m}_2,z})
     \widehat{\psi}_{\alpha\beta}(z),
    \end{array}
    \right.
\]
where $\widehat{\psi}_{\alpha\beta}(z):=\sum_{j=0}^{n-1}d_{\alpha\beta j}(\partial_{\overline{m}_2,z})
\widehat{\psi}_j(z)\in\CC[[z^{\frac{1}{\kappa}}]]_{s-\overline{s}}$ and $d_{\alpha\beta j}(\zeta)$ are some 
holomorphic functions of the variable $\xi=\zeta^{1/\kappa}$ and of polynomial growth. Hence, by Definition \ref{df:pseudo_1},  
$\widehat{u}=\sum_{\alpha=1}^l\sum_{\beta=1}^{n_{\alpha}}\widehat{u}_{\alpha\beta}$,
where $\widehat{u}_{\alpha\beta}=\Bo_{\Gamma_{-\overline{s}},z^{1/\kappa}}\widehat{v}_{\alpha\beta}$ satisfies (\ref{eq:gevrey})
with
\[ 
 \widehat{\varphi}_{\alpha\beta}=\Bo_{\Gamma_{-\overline{s}},z^{1/\kappa}}\widehat{\psi}_{\alpha\beta}=
 \Bo_{\Gamma_{-\overline{s}},z^{1/\kappa}}\sum_{j=0}^{n-1}d_{\alpha\beta j}(\partial_{\overline{m}_2,z})
\widehat{\psi}_j(z)=\sum_{j=0}^{n-1}d_{\alpha\beta j}(\partial_{m_2,z})
   \widehat{\varphi}_j(z)
\]
for $\beta=1,\dots,n_{\alpha}$ and $\alpha=1,\dots,l$.
\par
To find the Gevrey order of $\widehat{v}_{\alpha\beta}=\sum_{j=0}^{\infty}v_{\alpha\beta j}(z)t^j$ with respect to $t$, observe that by \cite[Lemma 2]{Mic7} and by Lemma \ref{le:estimation} there exists $r>0$ such that
\begin{align*}
\sup_{|z|<r}|v_{\alpha\beta j}(z)|&=m_1(0){j\choose \beta-1} \frac{\sup_{|z|<r}|\lambda_{\alpha\beta}^{j}(\partial_{\overline{m}_2,z})\psi_{\alpha\beta}(z)|}{m_1(j)}\\
&\leq
\widetilde{A}\widetilde{B}^j\frac{\Gamma_{\overline{q}_{\alpha}(s-\overline{s}+\overline{s}+s_2)}(j)}{\Gamma_{s_1}(j)}
\leq A B^j \Gamma_{\overline{q}_{\alpha}(s+s_2)-s_1}(j)
\end{align*}
for some $A,B<\infty$. It means that $\widehat{v}_{\alpha\beta}\in G_{s-\overline{s},1/\kappa}[[t]]_{\overline{q}_{\alpha}(s_2+s)-s_1}$.
Finally, by Proposition \ref{pr:prop2} we conclude that
$\widehat{u}_{\alpha\beta}=\Bo_{\Gamma_{-\overline{s}},z^{1/\kappa}}\widehat{v}_{\alpha\beta}\in
G_{s,1/\kappa}[[t]]_{\overline{q}_{\alpha}(s_2+s) - s_1}$ 
or, equivalently, $\widehat{u}_{\alpha\beta}\in\CC[[t,z^{\frac{1}{\kappa}}]]_{\overline{q}_{\alpha}(s_2+s)-s_1,s}$.
 \end{proof}

 \section{Analytic solutions}
In this section we study the analytic continuation properties of the sum of convergent formal power
series solutions of
\begin{equation}
   \label{eq:formal0}
    \left\{
    \begin{array}{l}
     (\partial_{m_1,t}-\lambda(\partial_{m_2,z}))^{\beta}v=0\\
     \partial_{m_1,t}^j v(0,z)=0\ \ (j=0,\dots,\beta-2)\\
     \partial_{m_1,t}^{\beta-1} v(0,z)=\lambda^{\beta-1}(\partial_{m_2,z})\varphi(z)\in\Oo_{1/\kappa}(D),
    \end{array}
    \right.
   \end{equation}
where $\lambda(\zeta)$ is a root of the characteristic equation of (\ref{eq:general}). It means that $\lambda(\zeta)$ is an analytic function of the variable $\xi=\zeta^{1/\kappa}$ for $|\zeta|\geq r_0$ and
$\lambda(\zeta)\sim \lambda\zeta^q$. During this section we assume that $m_1$ and $m_2$ are moment functions 
of orders $1/k_1,1/k_2>0$, respectively.

Repeating the proof of \cite[Lemma 4]{Mic7} we get the following representation of solution $v$ of 
(\ref{eq:formal0})
\begin{Lem}
 \label{le:integral0}
 Let $v$ be a solution of (\ref{eq:formal0})  and $1/k_1 \geq q/k_2$.
  Then $v$ belongs to the space $\Oo_{1,1/\kappa}(D^2)$ and is given by
  \[
   v(t,z)=\frac{t^{\beta-1}}{(\beta-1)!}\partial_t^{\beta-1}
   \frac{m_1(0)}{2\kappa\pi i}\oint_{|w|=\varepsilon}^{\kappa}\varphi(w)\int_{r_0e^{i\theta}}^{\infty(\theta)}E_{m_1}(t\lambda(\zeta))
    E_{\widetilde{m}_2}(\zeta^{1/\kappa} z^{1/\kappa})\frac{e_{m_2}(\zeta w)}{\zeta w}\,d\zeta\,dw,
  \]
  where $\theta\in (-\arg w-\frac{\pi}{2k_2}, -\arg w + \frac{\pi}{2k_2})$ and
  $\widetilde{m}_2(u)=m_2(u/\kappa)$.
\end{Lem}

We generalise \cite[Lemma 5]{Mic7} as follows 
 \begin{Lem}
 \label{le:integral}
 Let $\lambda(\zeta)\sim\lambda\zeta^{q}$ be a root of the characteristic equation of (\ref{eq:general})
 for $q=\mu/\nu$ with relatively prime numbers $\mu,\nu\in\NN$, where 
 $\lambda(\zeta)$ is an analytic
 function of the variable $\xi=\zeta^{1/\kappa}$ for $|\zeta|\geq r_0$
(for some $r_0>0$).
 Moreover, let $1/k_1 = q/k_2$, $K>0$ and $d\in\RR$. We assume that $v$ is a solution of
 \[
    \left\{
    \begin{array}{l}
     (\partial_{m_1,t}-\lambda(\partial_{m_2,z}))^{\beta} v=0\\
     \partial_{m_1,t}^j v(0,z)=\varphi_j(z)\in\Oo_{1/\kappa}(D)\ \ (j=0,\dots,\beta-1).
    \end{array}
    \right.
 \]
If $\varphi_j\in\Oo^{qK}_{1/\kappa}(\widehat{S}_{(d+\arg\lambda+2k\pi)/q})$ for $k=0,\dots,q\kappa-1$ and $j=0,\dots,\beta-1$, then
 $v\in\Oo^{K,qK}_{1,1/\kappa}(\widehat{S}_{d}\times \widehat{S}_{(d+\arg\lambda+2k\pi)/q})$
 for $k=0,\dots,q\kappa-1$.
 Moreover, if additionally $\varphi_j\in\Oo(D)$ for $j=0,\dots,\beta-1$, then
 $v\in\Oo^{K,qK}_{1,1/\kappa}(\widehat{S}_{d+2n\pi/\nu}\times \widehat{S}_{(d+\arg\lambda+2k\pi)/q})$
 for $k=0,\dots,q\kappa-1$ and $n=0,\dots,\nu-1$.
\end{Lem}
\begin{proof}
First, we consider the case $k_1, k_2>1/2$. By the principle of superposition of solutions of linear equations,
we may
assume that $v$ satisfies (\ref{eq:formal0}) with
$\varphi\in\Oo^{qK}_{1/\kappa}(\widehat{S}_{(d+\arg\lambda+2k\pi)/q}(\widetilde{\delta};\widetilde{r}))$ for
$k=0,\dots,q\kappa-1$ and for some $\widetilde{\delta},\widetilde{r}>0$.
Hence, by Lemma \ref{le:integral0}, the function $v\in\Oo_{1,1/\kappa}(D^2)$ has the integral representation 
 \begin{equation}
   \label{eq:v}
   v(t,z)=\frac{t^{\beta-1}}{(\beta-1)!}\partial_t^{\beta-1}
   \frac{m_1(0)}{2\kappa\pi i}\oint_{|w|=\varepsilon}^{\kappa}\varphi(w)k(t,z,w)\,dw,
  \end{equation}
  where $\varepsilon < \widetilde{r}$ and
  \[
   k(t,z,w):=\int_{r_0e^{i\theta}}^{\infty(\theta)}E_{m_1}(t\lambda(\zeta))
    E_{\widetilde{m}_2}(\zeta^{1/\kappa} z^{1/\kappa})\frac{e_{m_2}(\zeta w)}{\zeta w}\,d\zeta
  \]
  with $\theta\in (-\arg w-\frac{\pi}{2k_2}, -\arg w + \frac{\pi}{2k_2})$ and $\widetilde{m}_2(u)=m_2(u/\kappa)$.
 Now we consider the function
  \begin{gather}
  \label{eq:kernel}
  (t,z)\mapsto k(t,z,w) \quad \textrm{for every fixed} \quad w\in\CC\setminus\{0\}.
 \end{gather}
 Observe that by Definition \ref{df:moment} there exist constants $A_i$ and $b_i$ ($i=1,2,3$)
 such that
 $|E_{m_1}(t\lambda(\zeta))|\leq A_1 e^{b_1|t|^{k_1}|\zeta|^{k_1q}}$,
 $|E_{\widetilde{m}_2}(\zeta^{1/\kappa} z^{1/\kappa})|\leq A_2 e^{b_2|\zeta|^{k_2}|z|^{k_2}}$ and
 $|e_{m_2}(\zeta w)|\leq A_3 e^{-b_3|\zeta|^{k_2}|w|^{k_2}}$. 
 Hence, there exist $a,b>0$ such that for every fixed $w\in\CC\setminus\{0\}$ and
 for every $(t,z)\in\CC^2$ satisfying $|t|<a|w|^{q}$ and $|z|<b|w|$, we have
 \[
  |k(t,z,w)|\leq\int_{r_0}^{\infty}\widetilde{A}e^{s^{k_2}(b_1|t|^{k_1}+b_2|z|^{k_2}-b_3|w|^{k_2})}\,ds\leq
    \int_{r_0}^{\infty}\widetilde{A}e^{-\widetilde{b}s^{k_2}|w|^{k_2}}\,ds<\infty
 \]
with some positive constants $\widetilde{A},\widetilde{b}$. 
 Hence the function (\ref{eq:kernel}) belongs to the space
 $\Oo_{1,1/\kappa}(\{(t,z)\in\CC^2\colon |t|<a|w|^{q},\ |z|<b|w|\})$
 and
 the right-hand side of (\ref{eq:v}) is a well-defined
 holomorphic function of the variables $t$ and $\zeta=z^{1/\kappa}$ in a complex neighbourhood of the origin. 
 
 To show that $v\in\Oo_{1,1/\kappa}(\widehat{S}_{d}\times \widehat{S}_{(d+\arg\lambda+2k\pi)/q})$ for
 $k=1,\dots,q\kappa-1$,
 we deform the $\kappa$-fold circle $|w|=\varepsilon$ in the integral representation (\ref{eq:v}) of $v$ as in the proof of 
 \cite[Lemma 5]{Mic7}. Namely, we split these
circles into $2q\kappa$ arcs $\gamma_{2k}$ and $\gamma_{2k+1}$ ($k=0,\dots,q\kappa-1$), where
$\gamma_{2k}$ extends between points of argument
$(d+\arg\lambda+2k\pi)/q\pm\widetilde{\delta}/3$ and $\gamma_{2k+1}$ extends between
$(d+\arg\lambda+2k\pi)/q+\widetilde{\delta}/3$ and $(d+\arg\lambda+2(k+1)\pi)/q-\widetilde{\delta}/3
\mod 2q\kappa\pi $.
Finally, since $\varphi\in\Oo_{1/\kappa}(S_{(d+\arg \lambda+2k\pi)/q}(\widetilde{\delta}))$, we may
deform $\gamma_{2k}$ into a path $\gamma_{2k}^R$ along the ray $\arg w =(d+\arg\lambda+2k\pi)/q-\widetilde{\delta}/3$ to a point with modulus $R$ (which can be chosen arbitrarily large), then along the circle $|w|=R$ to the ray
$\arg w = (d+\arg\lambda+2k\pi)/q+\widetilde{\delta}/3$
and back along this ray to the original circle. So, we have
\[
v(t,z)=\frac{t^{\beta-1}}{(\beta-1)!}\partial_t^{\beta-1}v_1(t,z) +
\frac{t^{\beta-1}}{(\beta-1)!}\partial_t^{\beta-1}v_2(t,z),
\]
where
\[
 v_1(t,z):=\sum_{k=0}^{q\kappa-1}\frac{m_1(0)}{2\kappa\pi i}\int_{\gamma_{2k+1}}\varphi(w)
 k(t,z,w)\,dw
\]
and
\[
 v_2(t,z):=\sum_{k=0}^{q\kappa-1}\frac{m_1(0)}{2\kappa\pi i}\int_{\gamma_{2k}^R}\varphi(w)
 k(t,z,w)\,dw.
\]
To study the analytic continuation of $v_1$, observe that for $\arg t = d$,
$\arg z = (d+\arg\lambda+2k\pi)/q$ ($k=0,\dots,q\kappa-1$), $\arg w \neq (d+\arg\lambda+2k\pi)/q$ ($k\in\ZZ$)
and for $q=k_2/k_1$, we may choose a direction $\theta$ in (\ref{eq:kernel}), which satisfies the following conditions
 \begin{itemize}
  \item \[
         \arg t + 2k\pi+ \arg \lambda + q\theta \in \Big(\frac{\pi}{2k_1}, 2\pi - \frac{\pi}{2k_1}\Big)\
         \textrm{for some}\ k\in\ZZ
        \]
    (in this case, by Definition \ref{df:moment}, we have 
    $|E_{m_1}(t\lambda(\zeta))|\leq C|t\lambda(\zeta)|^{-1}$ as
    $\zeta\to\infty$, $\arg\zeta=\theta$),
  \item \[
         \arg z/\kappa + 2l\pi +\theta/\kappa \in \Big(\frac{\pi}{2k_2\kappa}, 2\pi - \frac{\pi}{2k_2\kappa}\Big)\
         \textrm{for some}\ l\in\ZZ
        \]
    (in this case, by Definition \ref{df:moment}, we have 
    $|E_{\widetilde{m}_2}(\zeta^{1/\kappa} z^{1/\kappa})|\leq C'|\zeta z|^{-1/\kappa}$ as
    $\zeta\to\infty$, $\arg\zeta=\theta$),
  \item
   \[
    \arg w + 2n\pi + \theta \in \Big(-\frac{\pi}{2k_2}, \frac{\pi}{2k_2}\Big)\
         \textrm{for some}\ n\in\ZZ
   \]
    (in this case, by Definition \ref{df:moment}, there exists $\varepsilon > 0$ such that
    \begin{gather*}
     \Big|\frac{e_{m_2}(\zeta w)}{\zeta w}\Big| \leq e^{-\varepsilon|\zeta|^{k_2}}\quad\textrm{as}\quad
     \zeta\to\infty,\quad\arg\zeta=\theta).
    \end{gather*}
 \end{itemize}

Hence there exist $\delta>0$ and $r>0$ such that the function
$v_1\in\Oo_{1,1/\kappa}(\widehat{S}_{d}(\delta;r)\times\widehat{S}_{(d+\arg\lambda+2k\pi)/q}(\delta;r))$
 for $k=0,\dots,q\kappa-1$.
 Moreover, there exists $C<\infty$ such that $|k(t,z,w)|<C$ for every $(t,z)\in\widehat{S}_{d}(\delta;r)\times
\widehat{S}_{(d+\arg\lambda+2k\pi)/q}(\delta;r)$ and for every $w\in \bigcup_{k=0}^{q\kappa-1}\gamma_{2k+1}$.
Hence
\[
 |v_1(t,z)| \leq  \frac{q\kappa}{2\kappa\pi}\max_{k=0,\dots,q\kappa-1}\int_{\gamma_{2k+1}}|\varphi(w)|C\,d|w|
  \leq  \widetilde{C}<\infty
\]
and we conclude that $v_1$ is bounded as $t\to\infty$ and $z\to\infty$.
\par
Now we are ready to study the analytic continuation of $v_2$, 
Since the function (\ref{eq:kernel}) belongs to the space
$\Oo_{1,1/\kappa}(\{(t,z)\in\CC^2\colon |t|<a|w|^{q},\ |z|<b|w|\})$, one can find $\delta,r>0$ such that 
$v_2\in\Oo_{1,1/\kappa}(\widehat{S}_{d}(\delta;r)\times\widehat{S}_{(d+\arg\lambda+2k\pi)/q}(\delta;r))$
 for $k=0,\dots,q\kappa-1$ as $R$ tends to infinity.
Estimating this integral we obtain
\[
 |v_2(t,z)| \leq  \frac{q\kappa}{2\kappa\pi}\max_{k=0,\dots,q\kappa-1}\int_{\gamma_{2k}^R}|\varphi(w)|C\,d|w|
  \leq  ARe^{BR^{qK}} \leq \widetilde{A} e^{\widetilde{B}_1|t|^K+\widetilde{B}_2|z|^{qK}},
\]
since $|t| \sim |w|^{q}=R^q$ and $|z|\sim |w|$.
\par
Hence also $v\in\Oo^{K,qK}_{1,1/\kappa}(\widehat{S}_{d}\times \widehat{S}_{(d+\arg\lambda+2k\pi)/q})$
 for $k=0,\dots,q\kappa-1$.
\par
In general case $k_1,k_2 > 0$, there exists $p\in\NN$ such that $\widetilde{k}_1:=pk_1 >1/2$ and 
$\widetilde{k}_2:=pk_2 >1/2$. By \cite[Lemma 3]{Mic7},
the function $w(t,z):=v(t^p,z^p)$ is a solution of
\[
  \left\{
    \begin{array}{l}
     (\partial^p_{\widetilde{m}_1,t} - \lambda(\partial^p_{\widetilde{m}_2,z}))^{\beta}w=0,\\
     \partial^{np}_{\widetilde{m}_1,t}w(0,z)=\varphi_n(z^p)
     \in\Oo^{pqK}_{1/\kappa}(\widehat{S}_{(d+\arg\lambda+2k\pi)/pq})\
     \textrm{for}\ n=0,\dots,\beta-1\\
     \partial^{j}_{\widetilde{m}_1,t}w(0,z)=0\ \textrm{for}\ j=1,\dots,\beta p-1\ \textrm{and}\ p\not|\ j,
    \end{array}
  \right.
\]
where $\widetilde{m}_1(u):=m_1(u/p)$ and $\widetilde{m}_2(u):=m_2(u/p)$ are moment functions of order $1/\widetilde{k}_1$ and $1/\widetilde{k}_2$ respectively.
\par
By Theorem \ref{th:gevrey} we conclude that $w=w_0+\cdots+w_{p-1}$ with $w_j$ ($j=0,\dots,p-1$) satisfying
\[
 \left\{
  \begin{array}{l}
   (\partial_{\widetilde{m}_1,t} - e^{i2j\pi/p}\lambda^{1/p}(\partial^p_{\widetilde{m}_2,z}))^{\beta}w_j=0,\\
   \partial^{n}_{\widetilde{m}_1,t}w_j(0,z)=\widetilde{\varphi}_{jn}(z)\in\Oo^{pqK}_{1/\kappa}(\widehat{S}_{(d+\arg\lambda+2k\pi)/pq})\
   \textrm{for}\ n=0,\dots,\beta-1.
  \end{array}
 \right.
\]
Applying the first part of the proof to the above equation we see that
$w_j(t,z)\in\Oo^{pK,pqK}_{1,1/\kappa}(\widehat{S}_{(d+2j\pi)/p}\times   
\widehat{S}_{(d+\arg\lambda+2k\pi)/pq})$ for $j=1,\dots,p$. It means that
$v(t,z)=w(t^{1/p},z^{1/p})\in\Oo^{K,qK}_{1,1/\kappa}(\widehat{S}_{d}\times 
\widehat{S}_{(d+\arg\lambda+2k\pi)/q})$ for $k=0,\dots,q\kappa-1$.
\par
To prove the last part of the lemma, observe that if
$\varphi_j\in\Oo^{qK}_{1/\kappa}(\hat{S}_{(d +\arg\lambda+2k\pi)/q})$ and
$\varphi_j\in\Oo(D)$ then also $\varphi_j\in\Oo^{qK}(\hat{S}_{(d +\arg\lambda+2k\pi)/q})$
and consequently 
$\varphi_j\in\Oo^{qK}(\hat{S}_{(d+2n\pi/\nu +\arg\lambda+2k\pi)/q})$ for $n=0,\dots,\nu-1$. Hence, replacing $d$
by $d+2n\pi/\nu$ we conclude that  
$v\in\Oo^{K,qK}_{1,1/\kappa}(\hat{S}_{d+2n\pi/\nu}\times \widehat{S}_{(d+\arg\lambda+2k\pi)/q})$
for $n=0,\dots,\nu-1$ and $k=0,\dots,q\kappa-1$.
\end{proof}

Now we are ready to generalise \cite[Theorem 3]{Mic7} as follows
\begin{Th}
   \label{th:main}
   Let $\lambda(\zeta)\sim\lambda\zeta^{q}$ be a root of the characteristic equation of (\ref{eq:general})
   for $q=\mu/\nu$ with relatively prime numbers $\mu,\nu\in\NN$, where 
   $\lambda(\zeta)$ is an analytic
   function of the variable $\xi=\zeta^{1/\kappa}$ for $|\zeta|\geq r_0$
   (for some $r_0>0$).
   Moreover, let us assume that $v$ is a solution of (\ref{eq:formal0}), $1/k_1 =q/k_2$, $K>0$ and $d\in\RR$.
   Then the following conditions are equivalent:
 \begin{enumerate}
   \item[(a)] $\varphi\in\Oo_{1/\kappa}^{qK}(\widehat{S}_{(d+\arg\lambda+2k\pi)/q})$ for $k=0,\dots,q\kappa-1$, 
   \item[(b)] $v\in\Oo^{K,qK}_{1,1/\kappa}(\widehat{S}_d\times \widehat{S}_{(d+\arg\lambda+2k\pi)/q})$
     for $k=0,\dots,q\kappa-1$.
   \item[(c)] $v\in\Oo_{1,1/\kappa}^K(\widehat{S}_d\times D)$,
   \item[(d)] $\partial^j_{m_2,z^{1/\kappa}}v(t,0)\in\Oo^K(\widehat{S}_d)$ for $j=0,\dots,q\kappa\beta-1$.
 \end{enumerate}
  If additionally we assume that $\varphi\in\Oo(D)$ then the above conditions are also equivalent to
 \begin{enumerate}
   \item[(e)] $v\in\Oo_{1,1/\kappa}^{K,qK}(\widehat{S}_{d+2n\pi/\nu} \times 
    \widehat{S}_{(d+\arg\lambda+2k\pi)/q})$ for $n=0,\dots,\nu-1$ and $k=0,\dots,q\kappa-1$,
   \item[(f)] $v\in\Oo_{1,1/\kappa}^K(\widehat{S}_{d+2n\pi/\nu}\times D)$ for $n=0,\dots,\nu-1$, 
 \end{enumerate}
\end{Th}
\begin{proof}
   The implication (a) $\Rightarrow$ (b) is given immediately by Lemma \ref{le:integral}.
   The implications (b) $\Rightarrow$ (c) and (c) $\Rightarrow$ (d) are trivial.
   To prove the implication (d) $\Rightarrow$ (a), observe that by \cite[Lemma 3]{Mic7}
   the function $w(t,z):=v(t^{q\kappa},z^{\kappa})$ satisfies 
 $$
 (\partial^{q\kappa}_{\widetilde{m}_1,t}-\lambda(\partial^{\kappa}_{\widetilde{m}_2,z}))^{\beta}w=0,
 $$
 where $\widetilde{m}_1(u):=m_1(u/q\kappa)$ and $\widetilde{m}_2(u):=m_2(u/\kappa)$ are moment functions
 of orders $1/\widetilde{k}_1:=1/k_1q\kappa$ and $1/\widetilde{k}_2:=1/k_2\kappa$.
 It means that $w$ is also a solution of the equation
 $$
  (\partial_{\widetilde{m}_1,t}-\widetilde{\lambda}_0(\partial_{\widetilde{m}_2,z}))^{\beta}\cdots
  (\partial_{\widetilde{m}_1,t}-\widetilde{\lambda}_{q\kappa-1}(\partial_{\widetilde{m}_2,z}))^{\beta} w=0,
 $$
 where 
 \begin{eqnarray*}
 \widetilde{\lambda}_j(\zeta):=e^{i2\pi j/q\kappa}\lambda^{1/q\kappa}(\zeta^{\kappa}) & \textrm{for} & j=0,\dots,q\kappa-1.
\end{eqnarray*}
Since $\widetilde{\lambda}_j(\zeta)$ is an analytic function for sufficiently large $|\zeta|$ with a pole order
equal to $1$ (more precisely $\widetilde{\lambda}_j(\zeta)\sim e^{i2\pi j/q\kappa}\lambda^{1/q\kappa}\zeta$) and
$1/\widetilde{k}_1=1/\widetilde{k}_2$, by \cite[Lemma 7]{Mic7} and by the condition (d),
the function $w$ satisfies also
   \begin{eqnarray*}
 \left\{
  \begin{array}{l}
   (\partial_{\widetilde{m}_2,z}-\widetilde{\lambda}^{-1}_0(\partial_{\widetilde{m}_1,t}))^{\beta}\cdots
   (\partial_{\widetilde{m}_2,z}-\widetilde{\lambda}^{-1}_{q\kappa-1}(\partial_{\widetilde{m}_1,t}))^{\beta}w=0,\\
   \partial_{\widetilde{m}_2,z}^{n}w(t,0)=\widetilde{\psi}_n(t)\in\Oo^{q\kappa K}(\hat{S}_{(d+2\pi k)/q\kappa})
  \end{array}
 \right.
\end{eqnarray*}
for $n=0,\dots,q\kappa\beta-1$ and $k=0,\dots, q\kappa-1$.
    Hence, by Theorem \ref{th:gevrey}, $w=w_0+\dots+w_{q\kappa-1}$ with $w_j$ ($j=0,\dots,q\kappa-1$) satisfying
    \begin{eqnarray*}
 \left\{
  \begin{array}{l}
   (\partial_{\widetilde{m}_2,z} - \widetilde{\lambda}^{-1}_j(\partial_{\widetilde{m}_1,t}))^{\beta}w_j=0,\\
   \partial^{n}_{\widetilde{m}_2,z}w_j(t,0)=\widetilde{\psi}_{jn}(t)\in\Oo^{q\kappa K}(\hat{S}_{(d+2\pi k)/q\kappa})
  \end{array}
 \right.
\end{eqnarray*}
  for $n=0,\dots,\beta-1$, $k=0,\dots,q\kappa -1$. Since $\widetilde{\lambda}_j^{-1}(\tau)\sim e^{-i2\pi j/q\kappa} \lambda^{-1/q\kappa} \tau$, 
    by Lemma \ref{le:integral} with replaced variables, we conclude that
    $w_j(t,z)\in\Oo^{q\kappa K}(D\times\hat{S}_{\theta_{jk}})$, where
    $$\theta_{jk}:=(d+2\pi k)/q\kappa - \arg(e^{-i2\pi j/q\kappa}\lambda^{-1/q\kappa})=(d+\arg\lambda+2\pi (k+j))/q\kappa$$ for $k=0,\dots,q\kappa-1$.
    In consequence, also $w(t,z)\in\Oo^{q\kappa K}(D\times\hat{S}_{(d+\arg\lambda+2\pi k)/q\kappa})$
    and finally $v(t,z)=w(t^{1/q\kappa},z^{1/\kappa})\in \Oo^{q K}_{1/q\kappa,1/\kappa}(D\times\hat{S}_{(d+\arg\lambda+2\pi k)/q})$. In particular $\varphi(z)\in \Oo^{q K}_{1/\kappa}(\hat{S}_{(d+\arg\lambda+2\pi k)/q})$ for $k=0,\dots,q\kappa-1$,
    which proves the implication (d) $\Rightarrow$ (a).
  
  If additionally $\varphi\in\Oo(D)$ then also
$\varphi\in\Oo^{qK}(\hat{S}_{(d+2n\pi/\nu +\arg\lambda+2k\pi)/q})$ for $n=0,\dots,\nu-1$. Hence, replacing $d$
by $d+2n\pi/\nu$ we conclude by Lemma \ref{le:integral} that
$v\in\Oo^{K,qK}_{1,1/\kappa}(\widehat{S}_{d+2n\pi/\nu}\times \widehat{S}_{(d+\arg\lambda+2k\pi)/q})$
for $n=0,\dots,\nu-1$ and $k=0,\dots,q\kappa-1$ an the implication (a) $\Rightarrow$ (e) holds. The last implications (e) $\Rightarrow$ (f) and (f) $\Rightarrow $ (c) are obvious.
\end{proof}

By the above theorem we conclude
\begin{Cor}
 \label{co:exp}
 If $K'>0$, $d'\in\RR$, $\varphi\in\Oo^{K'}(\widehat{S}_{d'})$ and $m$ is a moment function of order $0$,
 then also $\Bo_{m,z}\varphi\in\Oo^{K'}(\widehat{S}_{d'})$.
\end{Cor}
\begin{proof}
 Let $v$ be a solution of 
 \begin{gather*}
  (\partial_t-\partial_z)v=0,\quad v(0,z)=\varphi(z)\in\Oo^{K'}(\widehat{S}_{d'}).
 \end{gather*}
 Then $v(t,z)=\varphi(t+z)\in\Oo^{K'}(\widehat{S}_{d'}\times D)$. Since $m$ is a moment
 function of order $0$, we see that also $\Bo_{m,z}v\in\Oo^{K'}(\widehat{S}_{d'}\times D)$.
 On the other hand, by Proposition \ref{pr:formal}, $\Bo_{m,z}v$ is a solution of
 \begin{gather*}
   (\partial_t-\partial_{\Gamma_1 m,z})\Bo_{m,z}v=0,\quad \Bo_{m,z}v(0,z)=\Bo_{m,z}\varphi(z)\in\Oo(D).
 \end{gather*}
 Hence,
 applying Theorem \ref{th:main}, we conclude that $\Bo_{m,z}\varphi\in\Oo^{K'}(\widehat{S}_{d'})$.
\end{proof}

\section{Summable and multisummable solutions}
In this section we characterise summable formal solutions $\widehat{u}$ of (\ref{eq:gevrey}) in terms of the 
Cauchy data $\widehat{\varphi}$. Next, we also give a similar characterisation of multisummable normalised formal solutions
of general equation (\ref{eq:general}).

Applying Theorem \ref{th:main} we obtain the following impressive characterisation
of summable solutions of simple pseudodifferential equations (\ref{eq:gevrey})
\begin{Th}
\label{th:summable}
    Let $\lambda(\zeta)\sim\lambda\zeta^{q}$ be a root of the characteristic equation of (\ref{eq:general})
   for $q=\mu/\nu$ with relatively prime numbers $\mu,\nu\in\NN$, where 
   $\lambda(\zeta)$ is an analytic
   function of the variable $\xi=\zeta^{1/\kappa}$ for $|\zeta|\geq r_0$
   (for some $r_0>0$).
   We also assume that $m_1$, $m_2$ are moment functions of orders $s_1,s_2\in\RR$ respectively,
   $d,s\in\RR$, $s>-s_2$, $q>\frac{s_1}{s_2+s}$, $K=\big(q(s_2+s)-s_1\big)^{-1}$ and
   $\widehat{u}$ is a formal solution of
   \begin{equation}
    \label{eq:th_sum}
    \left\{
    \begin{array}{l}
     (\partial_{m_1,t}-\lambda(\partial_{m_2,z}))^{\beta}\widehat{u}=0\\
     \partial_{m_1,t}^j \widehat{u}(0,z)=0\ \ (j=0,\dots,\beta-2)\\
     \partial_{m_1,t}^{\beta-1} u(0,z)=\lambda^{\beta-1}(\partial_{m_2,z})\widehat{\varphi}(z)\in
     \CC[[z^{\frac{1}{\kappa}}]]_s.
    \end{array}
    \right.
   \end{equation}
   Then the following conditions are equivalent:
   \begin{enumerate}
     \item[(a)] $\Bo_{\Gamma_s,z^{1/\kappa}}\widehat{\varphi}\in\Oo^{qK}_{1/\kappa}
     (\widehat{S}_{(d+\arg\lambda+2k\pi)/q})$ 
     for $k=0,\dots,q\kappa-1$,
     \item[(b)] $\Bo_{\Gamma_{1/K},t}\Bo_{\Gamma_s,z^{1/\kappa}}\widehat{u}\in\Oo^K_{1,1/\kappa}
     (\widehat{S}_d\times D)$,  
     \item[(c)] $\Bo_{\Gamma_{1/K},t}\Bo_{\Gamma_{s},z^{1/\kappa}}u\in\Oo^{K,qK}_{1,1/\kappa}
     (\widehat{S}_d\times \widehat{S}_{(d+\arg\lambda+2k\pi)/q})$
       for $k=0,\dots,q\kappa-1$,
     \item[(d)] $\Bo_{\Gamma_{s_1/q-s_2},z^{1/\kappa}}\widehat{\varphi}$ is $qK$-summable in the directions
        $(d+\arg\lambda+2k\pi)/q$ for $k=0,\dots,q\kappa-1$,
     \item[(e)] $\widehat{u}(t,z)\in G_{s,1/\kappa}[[t]]$ is $K$-summable in the direction $d$,
   \end{enumerate}
   Moreover, if additionally $s>0$ and $qs_2 \geq s_1$ then the above conditions (a)--(e) are also equivalent to
   \begin{enumerate}
     \item[(f)] $\widehat{u}(t,z)\in\CC[[t,z^{\frac{1}{\kappa}}]]$ is $(K,1/s)$-summable in the directions
        $(d,(d+\arg\lambda+2k\pi)/q)$ for $k=0,\dots,q\kappa-1$,     
     \item[(g)] $\widehat{u}(t,z)\in\CC[[t,z^{\frac{1}{\kappa}}]]$ is $(K,1/s)$-summable in the directions
     \linebreak
        $O_{d,(d+\arg\lambda+2k\pi)/q}$ for $k=0,\dots,q\kappa-1$,  
    \end{enumerate}
\end{Th}

\begin{Rem}
   If we assume additionally that $\varphi\in\Oo(D)$ then we may replace the direction $d$ by $d+2n\pi/\nu$
   ($n=0,\dots,\nu-1$). Hence the conditions (a)--(e) are also equivalent to
   \begin{enumerate}
    \item[(h)] $\Bo_{\Gamma_{1/K},t}\Bo_{\Gamma_s,z^{1/\kappa}}\widehat{u}\in\Oo^K_{1,1/\kappa}(\widehat{S}_{d+2n\pi/\nu}
     \times D)$ for $n=0,\dots,\nu-1$,
     \item[(i)] $\Bo_{\Gamma_{1/K},t}\Bo_{\Gamma_{s},z^{1/\kappa}}u\in\Oo^{K,qK}_{1,1/\kappa}(\widehat{S}_{d+2n\pi/\nu} \times \widehat{S}_{(d+\arg\lambda+2k\pi)/q})$ for
     $n=0,\dots,\nu-1$ and $k=0,\dots,q\kappa-1$,
     \item[(j)] $\widehat{u}(t,z)\in G_{s,1/\kappa}[[t]]$ is $K$-summable in the directions $d+2n\pi/\nu$
        for $n=0,\dots,\nu-1$,
   \end{enumerate}
  and the conditions (f)--(g) are equivalent to
  \begin{enumerate}
   \item[(k)] $\widehat{u}(t,z)\in\CC[[t,z^{\frac{1}{\kappa}}]]$ is $(K,1/s)$-summable in the directions
        $(d+2n\pi/\nu,(d+\arg\lambda+2k\pi)/q)$ for $k=0,\dots,q\kappa-1$ and $n=0,\dots,\nu-1$.
   \item[(l)] $\widehat{u}(t,z)\in\CC[[t,z^{\frac{1}{\kappa}}]]$ is $(K,1/s)$-summable in the directions
     \linebreak
        $O_{d+2n\pi/\nu,(d+\arg\lambda+2k\pi)/q}$ for $k=0,\dots,q\kappa-1$ and $n=0,\dots,\nu-1$.
  \end{enumerate}
\end{Rem}

\begin{proof}[Proof of Theorem \ref{th:summable}.]
   First, observe that by Propositions \ref{pr:commutation} and \ref{pr:commutation2} the function $v:=\Bo_{\Gamma_{1/K},t}\Bo_{\Gamma_{s},z^{1/\kappa}}u$
   satisfies the equation
   \[
     \left\{
     \begin{array}{l}
       (\partial_{\overline{m}_1,t}-\lambda(\partial_{\overline{m}_2,z}))^{\beta}v=0\\
       \partial_{\overline{m}_1,t}^j v(0,z)=0\ \ (j=0,\dots,\beta-2)\\
       \partial_{\overline{m}_1,t}^{\beta-1} v(0,z)=
       \lambda^{\beta-1}(\partial_{\overline{m}_2,z})\Bo_{\Gamma_s,z^{1/\kappa}}\widehat{\varphi}(z)
       \in\Oo_{1/\kappa}(D),
     \end{array}
     \right.
   \]
   where $\overline{m}_1:=m_1\Gamma_{1/K}$ is a moment function of order
   $1/\overline{k}_1:= s_1 + 1/K = q(s_2+s)>0$ and $\overline{m}_2:=m_2\Gamma_s$
   is a moment function of order $1/\overline{k}_2:=s_2 + s>0$. Since $1/\overline{k}_1 = q/\overline{k}_2$,
   applying Theorem \ref{th:main} to $v$ we conclude that the properties (a)--(c) are equivalent.
   
   Moreover, by Remark \ref{re:summable} we obtain the equivalence
   (b) $\Leftrightarrow$ (e). 
   
   To show the equivalence between (a) and (d), observe that 
   $\Bo_{\Gamma_{s_1/q-s_2},z^{1/\kappa}}\widehat{\varphi}$ is $qK$-summable in directions
   $(d+\arg\lambda+2k\pi)/q$ for $k=0,\dots,q\kappa-1$ if and only if 
   $\Bo_{\Gamma_{1/qK},z^{1/\kappa}}\Bo_{\Gamma_{s_1/q-s_2},z^{1/\kappa}}\widehat{\varphi}\in
   \Oo^{qK}_{1,1/\kappa}(\widehat{S}_{(d+\arg\lambda+2k\pi)/q})$ for $k=0,\dots,q\kappa-1$.
   By Proposition \ref{pr:properties} and Corollary \ref{co:exp}, it is equivalent to (a).
   
   Now we assume additionally that $s>0$ and $qs_2 \geq s_1$.
   To find the equivalence between (f) and the previous conditions (a)--(e), it is sufficient to show 
   implications (c) $\Rightarrow$ (f) and (f) $\Rightarrow$ (b). To this end observe that $qK\leq 1/s$. Hence
   if $\Bo_{\Gamma_{1/K},t}\Bo_{\Gamma_{s},z^{1/\kappa}}u\in
   \Oo^{K,qK}_{1,1/\kappa}(\widehat{S}_d\times \widehat{S}_{(d+\arg\lambda+2k\pi)/q})$
   then also
  $\Bo_{\Gamma_{1/K},t}\Bo_{\Gamma_{s},z^{1/\kappa}}u\in
  \Oo^{K,1/s}_{1,1/\kappa}(\widehat{S}_d\times \widehat{S}_{(d+\arg\lambda+2k\pi)/q})$
   (for $k=0,\dots,q\kappa-1$) and consequently by Definition \ref{df:summable2} we conclude (f).
   In the opposite side, if $u$ satisfies (f) then
   $\Bo_{\Gamma_{1/K},t}\Bo_{\Gamma_{s},z^{1/\kappa}}u\in\Oo^{K,1/s}_{1,1/\kappa}
   (\widehat{S}_d\times \widehat{S}_{(d+\arg\lambda+2k\pi)/q})$. In particular,
   $\Bo_{\Gamma_{1/K},t}\Bo_{\Gamma_{s},z^{1/\kappa}}u\in\Oo^{K}(\widehat{S}_d\times D)$, which gives (b).
 
 Next we show the equivalence (c) $\Leftrightarrow$ (g). By Proposition \ref{pr:Bsum},
   $\widehat{u}(t,z)=\sum_{j,n=0}^{\infty}u_{jn}t^jz^{n/\kappa}$ is $(K,1/s)$ summable in the direction
   $O_{d,(d+\arg\lambda+2k\pi)/q}$ if and only if 
   \[
   \widetilde{v}(t,z):=\sum_{j,n=0}^{\infty}\frac{u_{jn}}{\Gamma(1+j/K+sn/\kappa)}t^kz^{n/\kappa} \in \Oo^{K,1/s}_{1,1/\kappa}(\widehat{S}_d
   \times \widehat{S}_{(d+\arg\lambda+2k\pi)/q}).
   \]
   So, it is sufficient to show
   \[
   v\in\Oo^{K,qK}_{1,1/\kappa}(\widehat{S}_d\times \widehat{S}_{(d+\arg\lambda+2k\pi)/q}) \Leftrightarrow
   \widetilde{v}\in \Oo^{K,1/s}_{1,1/\kappa}(\widehat{S}_d\times \widehat{S}_{(d+\arg\lambda+2k\pi)/q}).
   \]
   By Lemma \ref{le:beta} we get the following connection between
   $\widetilde{V}(t,z):=\widetilde{v}(t,z^{\kappa})$ and $V(t,z):=v(t,z^{\kappa})$
\[
   \widetilde{V}(t,z)=(1+\frac{1}{K}t\partial_t+\frac{s}{\kappa}z\partial_z)\int_0^1
   V(t\varepsilon^{1/K},z(1-\varepsilon)^{s/\kappa})\,d\varepsilon.
\]
By the above formula and by the assumption $Kq\leq 1/s$ we conclude that if
$v\in\Oo^{K,Kq}_{1,1/\kappa}(\widehat{S}_d \times \widehat{S}_{(d+\arg\lambda+2k\pi)/q})$
then $\widetilde{v}\in\Oo^{K,1/s}_{1,1/\kappa}(\widehat{S}_d \times \widehat{S}_{(d+\arg\lambda+2k\pi)/q})$.

To show the implication in the opposite side, we use the connection between the boundary 
conditions for $v$ and $\widetilde{v}$. Namely, since
\[
\partial_{\overline{m}_2,z^{1/\kappa}}^n \widetilde{v}(t,0)=
\frac{\overline{m}_2(n)}{\overline{m}_2(0)}\sum_{j=0}^{\infty}\frac{u_{jn}}{\Gamma(1+j/K+sn/\kappa)}t^j
\]
and
\[
\partial_{\overline{m}_2,z^{1/\kappa}}^n v(t,0)=
\frac{\overline{m}_2(n)}{\overline{m}_2(0)}\sum_{j=0}^{\infty}
\frac{u_{jn}}{\Gamma(1+j/K)\Gamma(1+sn/\kappa)}t^j,
\]
we get 
\[
\partial_{\overline{m}_2,z^{1/\kappa}}^n v(t,0) =
\Bo_{m_n',t}\partial_{\overline{m}_2,z^{1/\kappa}}^n \widetilde{v}(t,0),
\]
where $m_n'(u):=\frac{\Gamma(1+u/K+sn/\kappa)}{\Gamma(1+u/K)\Gamma(1+sn/\kappa)}$ is a moment function
of order $0$ for $n=0,\dots,q\kappa\beta-1$. So, since
$\partial_{\overline{m}_2,z^{1/\kappa}}^n \widetilde{v}(t,0)\in\Oo^{K}(\widehat{S}_d)$, by Corollary \ref{co:exp}
we see that also $\partial_{\overline{m}_2,z^{1/\kappa}}^n v(t,0)\in\Oo^{K}(\widehat{S}_d)$ for
$n=0,\dots,q\kappa\beta-1$. Hence, by Theorem \ref{th:main} we conclude that
$v\in\Oo^{K,qK}_{1,1/\kappa}(\widehat{S}_d \times \widehat{S}_{(d+\arg\lambda+2k\pi)/q})$.
\end{proof}

Now we return to the general equation (\ref{eq:general}).
For convenience we assume that
 \[
   P(\lambda,\zeta)=P_0(\zeta)\prod_{\alpha=1}^{\widetilde{n}}\prod_{\beta=1}^{l_{\alpha}}
   (\lambda-\lambda_{\alpha\beta}(\zeta))^{n_{\alpha\beta}},
 \]
  where $\lambda_{\alpha\beta}(\zeta)\sim\lambda_{\alpha\beta}\zeta^{q_{\alpha}}$ are the roots of the characteristic equation $P(\lambda,\zeta)=0$ with 
  pole orders $q_\alpha\in\QQ$ and leading terms $\lambda_{\alpha\beta}\in\CC\setminus\{0\}$ for
  $\beta=1,\dots,l_{\alpha}$ and $\alpha=1,\dots,\widetilde{n}$.
  \par
 We also assume that $s,s_1,s_2\in\RR$, $s_1>0$, $s+s_2>0$ and $\widehat{\varphi}_j\in\CC[[z]]_s$ for $j=0,...,n-1$.
 Without loss of generality we may assume that there exist exactly $N$ pole orders of the roots of the 
 characteristic equation,
 which are greater than $\frac{s_1}{s_2+s}$, say $\frac{s_1}{s_2+s}<q_1<\cdots<q_N<\infty$ and let
 $K_{\alpha}>0$ be defined by
 $K_{\alpha}:=(q_{\alpha}(s_2+s)-s_1)^{-1}$ for $\alpha=1,\dots,N$. 
 \par
 By Theorem \ref{th:gevrey}, the normalised formal solution $\widehat{u}$ of (\ref{eq:general}) is given by
  \begin{equation}
   \label{eq:decomposition}
   \widehat{u}=\sum_{\alpha=1}^{\widetilde{n}}\sum_{\beta=1}^{l_{\alpha}}\sum_{\gamma=1}^{n_{\alpha\beta}}
   \widehat{u}_{\alpha\beta\gamma}
  \end{equation}
  with $\widehat{u}_{\alpha\beta\gamma}$ satisfying
 \[
  \left\{
   \begin{array}{l}
    (\partial_{m_1,t}-\lambda_{\alpha\beta}(\partial_{m_2,z}))^{\gamma} \widehat{u}_{\alpha\beta\gamma}=0\\
    \partial_{m_1,t}^j \widehat{u}_{\alpha\beta\gamma}(0,z)=0\ \ \textrm{for}\ \ j=0,\dots,\gamma-2\\
    \partial_{m_1,t}^{\gamma-1}\widehat{u}_{\alpha\beta\gamma}=
    \lambda_{\alpha\beta}(\partial_{m_2,z})\widehat{\varphi}_{\alpha\beta\gamma}(z),
   \end{array}
  \right.
 \]
 where $\widehat{\varphi}_{\alpha\beta\gamma}(z)=\sum_{j=0}^{n-1}d_{\alpha\beta\gamma j}(\partial_{m_2,z})\widehat{\varphi}_j(z)\in\CC[[z^{\frac{1}{\kappa}}]]_s$
 and $d_{\alpha\beta\gamma j}(\zeta)$ are holomorphic functions of the variable $\xi=\zeta^{1/\kappa}$ of polynomial growth at infinity.
 \par
 Since $q_{\alpha}\leq \frac{s_1}{s_2+s}$ for $\alpha=N+1,\dots,\widetilde{n}$, by Theorem \ref{th:gevrey},
 $\widehat{u}_{\alpha\beta\gamma}$ is convergent for $\gamma=1,\dots,n_{\alpha\beta}$, $\beta=1,\dots,l_{\alpha}$
 and $\alpha=N+1,\dots,\widetilde{n}$.
 \par
 Under the above conditions, immediately by Theorem \ref{th:summable} we get (see also \cite[Theorem 5]{Mic7})
 \begin{Th}
  \label{th:multi1}
   Let $(d_1,\dots,d_N)\in\RR^N$ be an admissible multidirection
  with respect to $(K_1,\dots,K_N)$ and let $q_{\alpha}=\mu_{\alpha}/\nu_{\alpha}$ with relatively prime numbers
  $\mu_{\alpha},\nu_{\alpha}\in\NN$ for $\alpha=1,\dots,N$. We assume that
  \[
    \Bo_{\Gamma_s,z}\widehat{\varphi}_j(z)\in\Oo^{q_{\alpha}K_{\alpha}}
   (\widehat{S}_{(d_{\alpha}+\arg\lambda_{\alpha\beta}+2n_{\alpha}\pi)/q_{\alpha}})
  \]
  for every $j=0,\dots,n-1$,
   $n_{\alpha}=0,\dots,\mu_{\alpha}-1$, $\beta=1,\dots,l_{\alpha}$ and $\alpha=1,\dots,N$.
  Then the normalised formal solution $\widehat{u}\in G_{s,1/\kappa}[[t]]$ of (\ref{eq:general})
  is $(K_1,\dots,K_N)$-multisummable in the multidirection $(d_1,\dots,d_N)$.
 \end{Th}
 
 In general, the sufficient condition for the multisummability of $\widehat{u}$ given in Theorem \ref{th:multi1} is not necessary, since the multisummability of $\widehat{u}$ satisfying (\ref{eq:decomposition}) does
 not imply the summability of $\widehat{u}_{\alpha\beta\gamma}$ (see \cite[Example 2]{Mic7}).
 For this reason, following \cite{Mic7}, we define a kind of multisummability for which that implication holds.
 \begin{Df}
 Let $(d_1,\dots,d_N)$ be an admissible multidirection with respect to $(K_1,\dots,K_N)$.
 We say that \emph{$\widehat{u}$ is $(K_1,\dots,K_N)$-multisummable in the multidirection
 $(d_1,\dots,d_N)$ with respect to the decomposition (\ref{eq:decomposition})}
 if $\widehat{u}_{\alpha\beta\gamma}$ is $K_{\alpha}$-summable in the direction $d_{\alpha}$
 (for $\alpha=1,\dots,N$)
 and is convergent (for $\alpha=N+1,\dots,\widetilde{n}$), where $\beta=1,\dots,l_\alpha$ and $\gamma=1,\dots,n_{\alpha\beta}$.
 \end{Df}
 
 Repeating the proof of \cite[Theorem 6]{Mic7} with \cite[Theorem 4]{Mic7} replaced by Theorem \ref{th:summable}, we conclude
 \begin{Th}
  \label{th:multi2}
  Let $(d_1,\dots,d_N)\in\RR^N$ be an admissible multidirection
  with respect to $(K_1,\dots,K_N)$ and let $q_{\alpha}=\mu_{\alpha}/\nu_{\alpha}$ with relatively prime numbers
  $\mu_{\alpha},\nu_{\alpha}\in\NN$ for $\alpha=1,\dots,N$. We assume that $\widehat{u}$ is the normalised formal solution of
  \[
     \left\{
      \begin{array}{l}
       P(\partial_{m_1,t},\partial_{m_2,z})\widehat{u}=0\\
       \partial_{m_1,t}^j \widehat{u}(0,z)=0\ \ (j=0,\dots,n-2)\\
       \partial_{m_1,t}^{n-1} \widehat{u}(0,z)=\widehat{\varphi}(z)\in\CC[[z]]_s.
      \end{array}
      \right.
  \]
  Then $\widehat{u}\in G_{s,1/\kappa}[[t]]$ is $(K_1,\dots,K_N)$-multisummable in the multidirection $(d_1,\dots,d_N)$ with respect to the decomposition
  (\ref{eq:decomposition}) if and only if
  \[
   \Bo_{\Gamma_s,z}\widehat{\varphi}\in\Oo^{q_{\alpha}K_{\alpha}}
   (\widehat{S}_{(d_{\alpha}+\arg\lambda_{\alpha\beta}+2n_{\alpha}\pi)/q_{\alpha}})
  \]
  for every $n_{\alpha}=0,\dots,\mu_{\alpha}-1$, $\beta=1,\dots,l_{\alpha}$ and $\alpha=1,\dots,N$.
 \end{Th}
 
 \begin{Rem}
  Analogously, one can also consider the multisummability in two variables using the approaches given by Sanz or
  Balser. By Theorem \ref{th:summable} we obtain the same characterisation of multisummable solutions in two
  variables as in Theorems \ref{th:multi1} and \ref{th:multi2}.
 \end{Rem}

\section{An example}
In this section we give a simple example illustrating the developed theory. For fixed $q\in\NN$ and $s\in\RR$ we discuss the solution of the equation
\begin{gather}
 \label{eq:example}
 (\partial_t - \partial_z^q)\widehat{u}=0,\quad \widehat{u}(0,z)=\widehat{\varphi}(z)\in\CC[[z]]_s.
\end{gather}
Observe that $\widehat{u}$ satisfies equation $(\partial_{m_1,t}-\lambda(\partial_{m_2,z}))\widehat{u}=0$
with the moment functions $m_1=m_2=\Gamma_1$ and $\lambda(\zeta)=\zeta^q$. We have
\begin{Cor}
 \label{co:simple}
 Let $s\in\RR$, $q\in\NN$ and $\widehat{u}$ be a formal power series solution of (\ref{eq:example}). Then the following
 conditions are equivalent:
 \begin{enumerate}
  \item[1)] $\widehat{u}(0,z)\in\CC[[z]]_s$.
  \item[2)] $\widehat{u}(t,0)\in\CC[[t]]_{q(1+s)-1}$.
  \item[3)] $\widehat{u}(t,z)\in\CC[[t,z]]_{q(1+s)-1,s}$.
 \end{enumerate}
\end{Cor}
\begin{proof}
 The implications 3) $\Rightarrow$ 2) and 3) $\Rightarrow$ 1) are obvious. The implication 1) $\Rightarrow$ 3)
 follows from Theorem \ref{th:gevrey}. So, it is sufficient to show the implication 2) $\Rightarrow$ 3). To this end,
 observe that $\widehat{u}$ satisfies the equation
 \begin{gather*}
 (\partial_z -\lambda_1(\partial_t))\cdots(\partial_z -\lambda_q(\partial_t))\widehat{u}=0,
 \quad \widehat{u}(t,0)\in\CC[[t]]_{q(1+s)-1},
 \end{gather*}
 where $\lambda_n(\zeta)=e^{i2n\pi/q}\zeta^{1/q}$ for $n=1,\dots,q$.
 \par
 Hence, by Theorem \ref{th:gevrey} with replaced variables $t$ and $z$, we get $\widehat{u}=\widehat{u}_1+\cdots+\widehat{u}_q$,
 where $\widehat{u}_n$ satisfies the equation $(\partial_z -\lambda_n(\partial_t))\widehat{u}_n=0$ and
 $\widehat{u}_n\in\CC[[t,z]]_{q(1+s)-1,s}$ for $n=1,\dots,q$. It means that also $\widehat{u}\in\CC[[t,z]]_{q(1+s)-1,s}$.
\end{proof}

Assuming $s=0$ (resp. $s<0$) in Corollary \ref{co:simple}, replacing $\widehat{u}$ and $\widehat{\varphi}$ in
(\ref{eq:example}) by their sums $u$ and $\varphi$, and applying Remark \ref{re:convergent},
we obtain
\begin{Cor}
 The solution $u$ of (\ref{eq:example}) is $t$-analytic in a complex neighbourhood of the origin if and only
       if $\varphi\in\Oo(D)$ (for $q=1$) and $\varphi\in\Oo^{\frac{q}{q-1}}(\CC)$ (for $q=2,3,\dots$).
 Furthermore, the solution $u$ of (\ref{eq:example}) is $t$-entire of exponential growth of order $k>0$ if and only if
       $\varphi\in\Oo^{\frac{kq}{k(q-1)+1}}(\CC)$.
\end{Cor}

By Theorem \ref{th:summable} we obtain immediately
\begin{Prop}
Let $d\in\RR$, $\widehat{u}$ be a formal power series solution of (\ref{eq:example}) and $q(1+s)-1 > 0$. Then the following conditions are
equivalent:
\begin{enumerate}
\item[1.] $\widehat{u}\in G_{s,1}[[t]]$ is $(q(1+s)-1)^{-1}$-summable in the direction $d$.
\item[2.] $\Bo_{\Gamma_s,z}\widehat{\varphi}\in\Oo^{\frac{q}{q(1+s)-1}}(\widehat{S}_{(d+2k\pi)/q})$
          (for $k=0,\dots,q-1$).
\item[3.] $\Bo_{\Gamma_{1/q-1},z}\widehat{\varphi}$ is $\frac{q}{q(1+s)-1}$-summable in the directions
          $(d+2k\pi)/q$ for $k=0,\dots,q-1$.
\end{enumerate}
If additionally $s>0$ then the conditions 1.--3. are equivalent to
\begin{enumerate}
 \item[4.] $\widehat{u}\in \CC[[t,z]]$ is $((q(1+s)-1)^{-1},s^{-1})$-summable in the directions
           $(d,(d+2k\pi)/q))$ for $k=0,\dots,q-1$.
 \item[5.] $\widehat{u}\in \CC[[t,z]]$ is $((q(1+s)-1)^{-1},s^{-1})$-summable in the directions
           $O_{d,(d+2k\pi)/q}$ for $k=0,\dots,q-1$.
\end{enumerate}
\end{Prop}
\bibliographystyle{siam}
\bibliography{summa}
\end{document}